\documentclass[onefignum,onetabnum]{siamart171218}

\usepackage{lipsum}
\usepackage{amsfonts}
\usepackage{bbm}
\usepackage{graphicx}
\usepackage{algorithm, algpseudocode}
\usepackage{epstopdf}
\usepackage{cancel}
\usepackage[caption=false]{subfig}
\usepackage{enumitem}

\graphicspath{{figures/}}

\ifpdf
  \DeclareGraphicsExtensions{.eps,.pdf,.png,.jpg}
\else
  \DeclareGraphicsExtensions{.eps}
\fi

\newsiamremark{remark}{Remark}
\newsiamremark{hypothesis}{Hypothesis}
\crefname{hypothesis}{Hypothesis}{Hypotheses}
\newsiamthm{claim}{Claim}

\headers{Conditional Monte Carlo for Reaction Networks}{D. F.\ Anderson \and  K. W.\ Ehlert}

\title{Conditional Monte Carlo for Reaction Networks
}

\author{
David F. Anderson\thanks{Department of Mathematics, University of Wisconsin-Madison 
  (\email{anderson@math.wisc.edu}) }
  \and
  Kurt W. Ehlert\thanks{Department of Mathematics, University of Wisconsin-Madison 
  (\email{kehlert@math.wisc.edu}) }
  }

\usepackage{amsopn}

\usepackage{soul}

\newcommand{\Z}{\mathbb{Z}}

\newcommand{\tr}[1]{\text{tr}\left(#1\right)}
\DeclareMathOperator*{\argmin}{arg\,min}

\ifpdf
\hypersetup{
  pdftitle={Conditional Monte Carlo for Reaction Networks},
  pdfauthor={David F.\ Anderson \and  Kurt W.\ Ehlert}
}
\fi

\begin{document}

\maketitle

\begin{abstract}
Reaction networks are often used to model interacting species in fields such as biochemistry and ecology. When the counts of the species are sufficiently large, the dynamics of their concentrations are typically modeled via a system of  differential equations.  However, when the counts of some species are small, the  dynamics of the counts are  typically modeled stochastically via a discrete state,  continuous time Markov chain.

A key quantity of interest for such models is the probability mass function of the process at some  fixed time. Since paths of such models are relatively straightforward to simulate, we can estimate the probabilities by constructing an empirical distribution. However, the support of the distribution is often diffuse across a high-dimensional state space, where the dimension is equal to the number of species. Therefore generating an accurate empirical distribution can come with a large computational cost.

We present a new Monte Carlo estimator that fundamentally improves on the ``classical'' Monte Carlo estimator described above. It also preserves much of classical Monte Carlo's simplicity. The idea is basically one of conditional Monte Carlo. Our conditional Monte Carlo estimator has two parameters, and their choice critically affects the performance of the algorithm.  Hence, a  key contribution of the present work is that we demonstrate how to approximate optimal values for these parameters in an efficient manner. Moreover, we provide a central limit theorem for our estimator, which leads to approximate confidence intervals for its error.
\end{abstract}



\section{Introduction}\label{sec:intro}

Systems of interacting species appear often  in nature. To better understand the dynamics of such systems, we can model them as  reaction networks with deterministic or stochastic dynamics \cite{AKbook,gillespie2013perspective,karlebach2008modelling, wilkinson2006stochastic}. If the counts of the constituent species are high, then the dynamics are commonly modeled by a system of   differential equations \cite{AKbook, FeinbergLec79, wilkinson2006stochastic}. However, if the count of any species is small, then a stochastic model with a discrete state space is more appropriate \cite{AK2011b, AKbook, McAdams814, rao2002control, Thattai8614, wilkinson2006stochastic}.

Since the amount of each species is necessarily nonnegative and discrete, the state space of the stochastic process is a subset of $\Z_{\ge 0}^d$, where $d$ is the number of species types. Let $\nu$ be the distribution of the initial state, which is often a point mass distribution, and suppose we are interested in the distribution of the state of the process at some fixed time $t>0$. That is, if $X(t)$ is the state of the process at time $t$, then we would like to know the value of
\begin{equation*}
p_t^\nu (x) \stackrel{\text{def}}{=}  P_\nu(X(t)=x),\, x\in\Z_{\ge 0}^d.
\end{equation*}

In general, finding the exact values of $p_t^\nu(\cdot)$ is extremely difficult.  More precisely, the authors are not aware of any general class of models for which $p_t^\nu$ can be solved for explicitly, with the exception of linear, or first-order, models \cite{jahnke2007solving} or, more generally, models that satisfy a dynamical and restricted complex-balanced condition and admit a time-dependent product form Poisson distribution \cite{ASY2019}. However, there are many numerical methods that give an estimate. One type of approach is to approximately solve Kolmogorov's forward equation, which is  called the chemical master equation (CME) in much of the biology and chemistry literature. The CME can be written as
\begin{equation}\label{eq:cme}
\frac{d}{dt} p^\nu_t(x) = \sum_{r=1}^R \big[p_t^\nu(x-\zeta_r)\lambda_r(x-\zeta_r) - p_t^\nu(x)\lambda_r(x)\big],\, x\in\Z_{\ge 0}^d,
\end{equation}
where $R$ is the number of reactions in the system, $\lambda_r: \mathbb{Z}^d_{\ge 0}\to \mathbb{R}_{\ge 0}$ is the intensity (or propensity) function for the $r$th reaction,  $\zeta_r \in \Z^d$ gives the net change in the counts of the species due to an occurrence of the $r$th reaction, and the initial distribution $p_0^\nu(\cdot)$ is given by $\nu$.  See \cref{sec:mathematical_model} for the precise specification of the model, including terminology.

For most models of interest, solving \eqref{eq:cme} entails solving a  high-dimensional (often infinite-dimensional) system of linear ordinary differential equations. Solving such a system directly is almost always very difficult, so there has been a considerable amount of research devoted to the development of fast and accurate approximate algorithms. The general approach for many such algorithms is to first truncate the state space of the system to a smaller subset.  This truncation makes solving the problem computationally feasible, at the cost of introducing a controllable error to the solution. 
After truncation, the new system of ODEs must be solved. 

There is currently a wide variety of methods for performing both the truncation step and solution step.  In particular, there is the finite state projection algorithm \cite{munsky2006finite, vo2016improved}, the uniformization method \cite{didier2009fast}, sliding window methods \cite{henzinger2009sliding, wolf2010solving}, the sparse grid method \cite{hegland2007solver},  the radial basis function approximation \cite{kryven2015solution}, a class of spectral methods \cite{engblom2009spectral, jahnke2010solving}, and methods that specialize to systems with multiple scales  \cite{busch2006numerical, cao2016accurate, macnamara2008stochastic, macnamara2008multiscale, pelevs2006reduction}.  Moreover, there are  tensor methods \cite{kazeev2014direct, sidje2015solving, vo2017adaptive} that  represents  the truncated CME with tensors.  

As an alternative to approximating \eqref{eq:cme} directly via the methods above,  we can take a Monte Carlo approach.  That is, we can generate $n$  independent and identically distributed (i.i.d.) realizations of the process $X$, denoted by $\{X_i\}_{i=1}^n$, and use the Monte Carlo estimator	
\begin{equation}\label{eq:classical_mc_estimator}
\frac{1}{n} \sum_{i=1}^n \mathbbm{1}(X_i(t)=x) \approx \mathbb{E}_{\nu,0}\left[\mathbbm{1}(X(t)=x)\right] = p_t^\nu(x),
\end{equation}
where $\mathbb{E}_{\nu,0}$ is the expectation under the initial distribution $\nu$ and starting time of zero.
By the strong law of large numbers, the approximation becomes an equality as $n$ goes to infinity.

To utilize the above estimator, we need to simulate exact realizations of the process $X$ over the time interval $[0,t]$, and there are many methods to choose from. In particular, there is the Gillespie algorithm, also called the stochastic simulation algorithm, \cite{gillespie1976general}, the next reaction method \cite{gibson2000efficient}, and the modified next reaction method \cite{anderson2007modified}, which are all straightforward to implement and often have similar efficiency.  For our numerical results in the later sections, we used the modified next reaction method. 

One drawback of using the Monte Carlo estimator \cref{eq:classical_mc_estimator} to approximate the solution to the CME \eqref{eq:cme} is that  huge numbers of simulations are generally required to achieve a high level of accuracy. That said, the Monte Carlo estimator has at least two distinct advantages when compared against the methods that approximately solve the CME directly: it is very simple to implement and it is substantially less sensitive to the dimension of the state space.

There are two natural ways to improve upon a Monte Carlo estimator.  The first way is to decrease the time required to generate realizations of the random samples (i.e., the process $X$ in our case). Lowering the time required to generate paths of the  processes that we are interested in has been an active area of research for almost two decades \cite{anderson2007modified, gibson2000efficient, mauch2011efficient, mccollum2006sorting, ramaswamy2009new, slepoy2008constant}. Moreover, researchers have also designed efficient algorithms that generate approximate paths that trade some accuracy for speed \cite{anderson2008incorporating, AK2011,cao2006efficient, cao2007adaptive, ehlert2014lazy, gillespie2001approximate, haseltine2002approximate, rathinam2003stiffness}.

The second way to improve upon a Monte Carlo estimator, and the focus of this article, is to instead lower the variance of the estimator itself.  There are many broadly applicable variance reduction techniques, including coupling methods, control variates, stratified sampling, antithetic random variables, quasi-Monte Carlo, and conditional Monte Carlo \cite{GlassermannBook,mcbookOWEN}.

In this paper, we utilize a form of conditional Monte Carlo to reduce the variance.  Briefly, conditional Monte Carlo follows from the observation that for one-dimensional random variables $X$ and $Y$, defined on the same probability space, we have  $E[X] = E[E[X|Y]]$, and  $\text{Var}(E[X|Y]) \le \text{Var}(X)$, so long as all the expectations are well defined \cite{ASV2017}.  That is, one can always reduce variance by conditioning.  Of course, the ``art'' is in the selection of an appropriate random variable $Y$.

Returning to our situation, define $\mathbb{E}_{\nu,s}[f(X(t)]$ as the expectation of $f(X(t))$ taken with respect to the initial state distribution $\nu$ and starting time $0 \le s \le t$.  That is, $P(X(s) = x) = \nu(x)$.  If $\nu$ is a point-mass distribution at $y \in \Z^d_{\ge0}$, then we  write $\mathbb{E}_{y, s}[f(X(t))]$. Fix $h\in[0,t]$, then
\begin{alignat}{2}
	p^\nu_t(x) &= \mathbb{E}_{\nu,0}\left[{\mathbbm{1}(X(t)=x)}\right]\nonumber\\
	&= \mathbb{E}_{\nu,0}\left[\mathbb{E}_{\nu,0}\left[\left.\mathbbm{1}(X(t)=x)\right| X(t-h)\right]\right]\nonumber\\
	&= \mathbb{E}_{\nu,0}\left[\mathbb{E}_{X(t-h),t-h}\left[\mathbbm{1}(X(t)=x)\right]\right] \,\,\, && \text{(Markov property)}\nonumber\\
	&= \lim_{n\to\infty} \frac{1}{n} \sum_{i=1}^n \mathbb{E}_{X_i(t-h),t-h}\left[\mathbbm{1}(X(t)=x) \right]\text{, a.s.} \quad && \text{(strong law of large numbers)}\label{def:estimator_with_cond}
\end{alignat}
where the $\{X_i(t-h)\}_{i=1}^n$ are i.i.d.\ realizations of $X(t-h)$. A natural estimator for the right hand side of the above equation is
\begin{equation}\label{def:conditional_mc_estimator}
\hat{p}_t^\nu(x;n,m,h) \stackrel{\text{def}}{=}  \frac{1}{n} \sum_{i=1}^n \frac{1}{m} \sum_{j=1}^m \mathbbm{1}(X_{ij}(t)=x),
\end{equation}
where we generate the $X_{ij}$ in the following manner:
\begin{itemize}
\item simulate $n$ independent realizations of the process $X$ over the time interval $[0,t-h]$, each with an initial value determined by  $\nu$, and denote the $i$th realization by $X_i$,
\item for each $i \in \{1,\dots,n\}$, generate $m$ conditionally independent realizations over the time interval $[t-h,t]$, each of which has  initial state $X_i(t-h)$.  Denote the $j$th such realization by $X_{ij}$.
\end{itemize}
Note that for each $j \in \{1,\dots,m\}$, the process $X_{ij}$ is equal to $X_i$ over the interval $[0,t-h]$.   See \cref{figure1}. 

Since $\{X_{i_1 j}(t)\}_{j=1}^m$ and $\{X_{i_2 j}\}_{j=1}^m$ are independent for $i_1 \neq i_2$, the strong law of law numbers implies that with probability one we have
\begin{equation*}
\lim_{n\to\infty} \hat{p}_t^\nu(x;n,m,h) = \mathbb{E}_{\nu,0}\left[ \frac{1}{m} \sum_{j=1}^m \mathbbm{1}(X_{ij}(t)=x)\right] = p_t^\nu(x).
\end{equation*}
\begin{figure}\label{figure1}
	\centering
    \subfloat[Two independent realizations of the process over the time interval ${[0,2]}$.\label{fig:classical_mc_paths}]{{\includegraphics[width=0.35\textwidth]{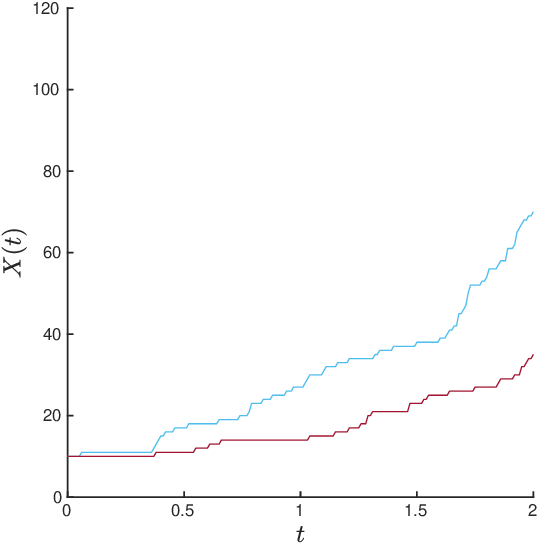}}}	
    \qquad
    \subfloat[Two independent realizations of the process generated over ${[0,1.5]}$.  Each is then followed by $m$ conditionally  independent ``branches''  simulated over ${[1.5, 2]}$.\label{fig:conditional_mc_paths}]{{\includegraphics[width=0.35\textwidth]{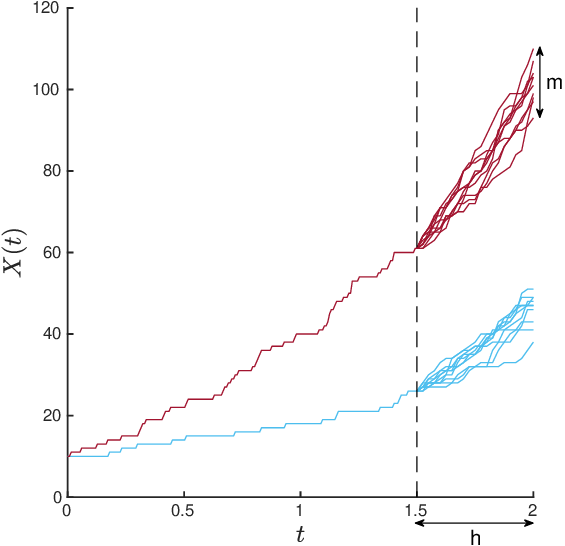}}}
    \caption{Paths generated for the birth model $X\rightarrow 2X$.}
\end{figure}
Hereafter we will refer to the original estimator \cref{eq:classical_mc_estimator} as \textit{classical Monte Carlo}, and the new estimator \cref{def:conditional_mc_estimator} as \textit{conditional Monte Carlo}.  The conditional Monte Carlo estimator has two unspecified parameters, denoted $m$ and $h$. The number of branches is determined by $m$, and the time at which branching occurs is controlled by $h$. If $m$ and $h$ are fixed, then the remaining parameter $n$ is simply chosen large enough such that the estimator's variance is below some desired threshold.  If $m=1$, $h=0$, or $h=t$, then the conditional and classical Monte Carlo estimators are the same. If $m > 1$ and $h\in(0,t)$, then for the same computational cost as classical Monte Carlo,  the conditional Monte Carlo estimator obtains more observations of $X(t)$. 
We would like to choose the values of $m$ and $h$ such that, in some sense, our new estimator is more efficient than classical Monte Carlo. In \cref{sec:optimization}, we provide an algorithm for finding optimal values of $m$ and $h$, which is the key contribution of this article.

The distributions produced by our conditional Monte Carlo method can, of course, be used to construct unbiased estimates of moments and other expectations.  However,  we stress here that our new estimator is optimized for estimating the \textit{entire distribution} of the process and not for estimating expectations.  Estimating expectations is a separate--and very important--problem that has seen a large amount of research activity over the past decade (see \cite{AndersonGangulyKurtz, AndHigham2012, AHS2014, AHS2017, AK2012,cao2006efficient,cao2007adaptive, MTV2016a} for a subset of  works focusing on this problem).  In fact, in  \cref{sec:expectations} we prove that the type of conditioning we carry out here (optimized for estimating the entire distribution) can \textit{not} be more efficient than standard Monte Carlo for the estimation of the expected value of a linear birth process at some future time $t>0$.  This may seem surprising at first since conditioning always reduces the variance (as discussed above).  However, in the present method we also use Monte Carlo to solve for the conditional expectation, which has its own cost. Determining better, and perhaps optimal, ways to estimate expectations via conditional Monte Carlo in the present context is a worthy direction of future research and will be discussed further in \cref{sec:conclusion} and \cref{sec:expectations}.

The remainder of the article is organized as follows. In \cref{sec:mathematical_model}, we define the continuous time Markov chain model of reaction networks. Then in \cref{sec:optimization}, we present an algorithm for finding the optimal values of $m$ and $h$, and also the full algorithm, \cref{alg:cond_mc}, for the implementation of the conditional Monte Carlo estimator. Next, in \cref{sec:numerical_results}, we  give  numerical results demonstrating the order of magnitude improvement that can be obtained with the use of conditional Monte Carlo in the current context.  In \cref{sec:clt}, we derive a central limit theorem for the error of the conditional Monte Carlo estimator and then test it on examples. Finally, in \cref{sec:conclusion}, we summarize our results and suggest ideas for future work. The proofs of the main results are in \cref{sec:additional_proofs}. The supplementary material contain more figures related to numerical results. An example MATLAB implementation of the conditional Monte Carlo algorithm is at \url{https://github.com/kehlert/conditional_monte_carlo_example}.

\section{Mathematical model}\label{sec:mathematical_model}

Suppose our reaction network has $d$ types of species and $R$ reactions. For $1 \le r\le R$, 
\begin{itemize}
\item[(i)]  we will denote by $\zeta_r$  the reaction vector for the $r$th reaction, meaning that if the $r$th reaction occurs at time $t$, and the process is currently in state $x\in \Z^d_{\ge 0}$, then the new state becomes $x + \zeta_r$;
\item[(ii)] we will denote by $\lambda_r:\Z^d_{\ge 0} \to [0,\infty)$ the intensity, or propensity, function of the $r$th reaction.
\end{itemize}
A standing assumption is that $\lambda_r(x) = 0$ if $x + \zeta_r \notin \Z^d_{\ge 0}$, which preserves the non-negativity of the components.  
We  let $X$ be  a continuous time Markov chain (CTMC) whose transition rate from state $x$ to $x'$ is
\begin{equation*}
q(x,x') = \sum_{r =1}^R \lambda_r(x) \mathbbm{1}(x' - x = \zeta_r).
\end{equation*}
Hence, $X$ is a Markov process with infinitesimal generator $Af(x) = \sum_{r = 1}^R \lambda_r(x)(f(x+\zeta_r) - f(x)),$ where $f: \mathbb{Z}_{\ge 0}^d \to \mathbb{R}$ is a bounded function with compact support.
We will denote our process by $X$, so that $X(t)\in\Z^d_{\ge 0}$ is the vector whose $i$th component gives the count of species $i$ at time $t \ge 0$.

The most common choice of intensity function is stochastic mass action kinetics.  Suppose that we require $y_i$ copies of species $i$ for the $r$th reaction to occur.   Then we say that $\lambda_r$ has stochastic mass action kinetics if
\begin{equation}\label{def:mass_action}
\lambda_r(x) = \kappa_r  \prod_{i=1}^d \frac{x_i !}{(x_i-y_i) !} \mathbbm{1}(x_i \ge y_i),
\end{equation}
for some  $\kappa_r>0$, which is called the rate constant of the reaction. For example, for the reaction $2A + B \to A + C$, where $A$, $B$, and $C$ are the species types in our model system, the reaction vector is $(-1,-1,1)^T$ and $y = (2,1,0)^T$, in which case $\lambda_r(x) = \kappa_r x_1(x_1-1)x_2,$ where we have ordered the species alphabetically.

None of our theoretical results assume that the $\lambda_r$ has the above mass action form, but the models we tested do use it unless otherwise noted.

One well--known representation for the stochastic process $X$ is the \textit{random time change representation} of Thomas Kurtz \cite{AK2011b, AKbook,kurtz1980representations}
\begin{equation}\label{def:ctmc_random_time_change}
X(t) = X(0) + \sum_{r=1}^R Y_r\left(\int_0^t \lambda_r\left(X(s)\right)\,ds\right)\zeta_r,
\end{equation}
where $X(0)$ is the initial state and the $Y_r$ are independent unit-rate Poisson processes. We will  make use of the above representation in some of our proofs.

\subsection{Examples}
\label{sec:examples}
In the subsequent sections, we intersperse numerical results, and below is a list of all the example models we used. The species to the left of the arrows are the reactants (giving the counts of the species consumed in the reaction), and those to the right are the products. The numbers above the arrows are the rate constants $\kappa_r$. Unless otherwise noted,  for every model and reaction we define the intensities $\lambda_r$ with \cref{def:mass_action}.
 
\begin{enumerate}[label=(\roman*)]
\item
\textit{Birth}

The initial state is $X(0)=10$ and $t=2$. The single reaction is
\begin{equation*}
X \xrightarrow{1} 2X.
\end{equation*}
Following \cref{def:mass_action}, the rate of the reaction is $\lambda(x) = x$.

\item
\textit{Birth--Death}

The initial state is $X(0)=100$ and $t=2$. There are two reactions
\begin{equation*}
\emptyset \xrightarrow{50} X,\, X \xrightarrow{1} \emptyset.
\end{equation*}
Following \cref{def:mass_action}, the rates of the reactions are $\lambda_1(x) = 50,$ and $\lambda_2(x) = x$, respectively.

\item
\textit{Lotka--Volterra}

This model is also often called the predator-prey model. The initial state is $A(0)=200$ and $B(0)=100$. We set $t=4$. The reactions are
\begin{equation*}
A \xrightarrow{2} 2A,\, A+B \xrightarrow{0.01} 2B, \, B \xrightarrow{2} \emptyset.
\end{equation*}
Following \cref{def:mass_action}, and after  ordering the species as $(A,B)$,  the rates of the reactions are $\lambda_1(x) = 2x_1$, $\lambda_2(x) = 0.01 x_1x_2$, and $\lambda_3(x) = 2x_2$, respectively.

\item
\textit{Dimerization}

In this model, $mRNA$ is translated into the protein $P$, which then dimerizes into $D$, and the dimer $D$ accumulates over time. The initial state for every species is zero except for $G(0)=1$. We set $t=1$. The reactions are
\begin{align*}
&G \xrightarrow{25} G+mRNA,\, mRNA \xrightarrow{100} mRNA + P\\
&2P \xrightarrow{0.001} D, \, mRNA \xrightarrow{0.1} \emptyset,\, P \xrightarrow{1} \emptyset.
\end{align*}
Following \cref{def:mass_action}, and after ordering the species as $(G,mRNA,P,D)$, the rates of the reactions are $\lambda_1(x) = 25x_1$, $\lambda_2(x) = 100 x_2$,  $\lambda_3(x) = 0.001 x_3(x_3-1)$, $\lambda_4(x) = 0.1x_2$, and $\lambda_5(x) = x_3$ respectively.

\item
\textit{Toggle}

Each species represses the production of the other, which leads to a  probability mass function that is multimodal. The initial state is $A(0)=B(0)=0$. We set $t=100$. The reactions are
\begin{equation*}
\emptyset \xrightarrow{} A,\, A \xrightarrow{} \emptyset, \, \emptyset \xrightarrow{} B, \, B \xrightarrow{} \emptyset.
\end{equation*}
For this model, the first and third  intensity functions are not chosen to be mass action.  Specifically, we let 
\[
	\lambda_1(x) = \frac{50}{1+2x_2}, \ \lambda_2(x) = x_1, \ \lambda_3(x) = \frac{50}{1 +2x_1}, \ \lambda_4(x) = x_2,
\]
where we again ordered the species as $(A,B)$.

\item
\textit{Fast/Slow}

$A$ and $B$ quickly convert into one another, and $B$ slowly turns into $C$. The initial state is $A(0) = B(0) = 100$ and $C(0) = 0$. We set $t=10$. The reactions are
\begin{equation*}
A \xrightarrow{10} B,\, B \xrightarrow{10} A, \, B \xrightarrow{0.1} C.
\end{equation*}
Following \cref{def:mass_action}, and after ordering the species as $(A,B,C)$, the rates are $\lambda_1(x) = 10x_1,$ $\lambda_2(x) = 10x_2$, and $\lambda_3(x) = 0.1x_2$, respectively.
\end{enumerate}

\section{Determining the values of \texorpdfstring{$m$}{\textit{m}} and \texorpdfstring{$h$}{\textit{h}} via optimization}\label{sec:optimization}

The conditional Monte Carlo estimator \cref{def:conditional_mc_estimator} is of little value without knowledge of which values of $m$ and $h$ to use. In this section, we will show that appropriate values can be found by numerically solving an easy optimization problem.

Recall that the distribution of the process is denoted by $p_t^\nu$, and we denote an estimate of this distribution by $\hat p_t^\nu$.  We will measure the quality of the estimation  via the  \textit{mean integrated squared error} (MISE), which is
\begin{equation}\label{def:MISE}
	\text{MISE}(\hat p_t^\nu) \stackrel{\text{def}}{=} \mathbb{E}_{\nu,0}\left[\sum_{x\in\mathbb{Z}_{\ge 0}^d} \!\! \big( \hat{p}_t^\nu(x)-p_t^\nu(x)\big)^2\right].
\end{equation}
Note that if $\hat p_t^{\nu}$ is constructed via our conditional Monte Carlo estimator, then  it, and by extension $\text{MISE}(\hat p_t^\nu)$, is a function of $n, m$, and $h$. 
Suppose we have a fixed computational budget, which we denote as $b$. We then want to choose the values of $n$, $m$, and $h$ so that we minimize  MISE$(\hat p_t^\nu)$ subject to our budget constraint $b$. We choose the squared error in \eqref{def:MISE}, as opposed to the total variation norm or some other $L^p$ error, as this choice was more amenable to analysis, especially in the derivation of the central limit theorem in Section \ref{sec:clt}.

\subsection{Computational cost model}\label{sec:comp_cost_model}

Assuming that our model is non-explosive\footnote{A process is said to explode if there are   an infinite number of transitions in a finite amount of time. A process is said to be non-explosive if the probability of an explosion is zero for all initial distributions  \cite{ACKK2018,NorrisMC97}.}  the expected number of reactions required to generate $\{X_{1j}\}_{j=1}^m$ is given by
\begin{equation*}
\underbrace{\mathbb{E}_{\nu,0}\left[\int_0^{t-h}\!\!\!\!\!\! \lambda_0(X(s))\,ds\right]}_{\text{expected \# of reactions in } [0,t-h]} \! + \,\, m\cdot \!\!\!\!\! \underbrace{\mathbb{E}_{\nu,0}\left[\int_{t-h}^t \!\!\!\! \lambda_0(X(s))\,ds\right]}_{\text{expected \# of reactions in } [t-h,t]}\!\!\!\!,
\end{equation*}
where $\lambda_0(x) = \sum_{r=1}^R \lambda_r(x)$ (see \cref{thm:expected_number_reactions}).  Hence, the expected computational cost for our conditional Monte Carlo estimator is
\begin{equation}\label{def:comp_cost_model}
n\cdot c\left(\mathbb{E}_{\nu,0}\left[\int_0^{t-h} \!\!\!\!\!\! \lambda_0(X(s))\,ds\right] + m\cdot \mathbb{E}_{\nu,0}\left[\int_{t-h}^t \!\!\!\! \lambda_0(X(s))\,ds\right]\right),
\end{equation}
where $c > 0$ is an unknown constant.

Since we cannot generally evaluate the expectations in the cost model \cref{def:comp_cost_model}, as this would be as difficult as the problem we are attempting to solve, we need to estimate them. To do so, fix a relatively small $\tilde{n}$  and  simulate $\tilde{n}$ i.i.d.\ paths $\{X_i\}_{i=1}^{\tilde{n}}$. Then the expectations are approximately equal to
\begin{equation}\label{eqn:total_intensity_integral_estimate}
\frac{1}{\tilde{n}} \sum_{i=1}^{\tilde{n}} \int_0^{t-h}\!\!\!\!\!\! \lambda_0(X_i(s))\,ds,\,\text{and } \frac{1}{\tilde{n}} \sum_{i=1}^{\tilde{n}} \int_{t-h}^t \!\!\!\! \lambda_0(X_i(s))\,ds.
\end{equation}
Importantly, for the fixed set of $\tilde n$ paths, the values \cref{eqn:total_intensity_integral_estimate} can be computed for a variety of different $h$ values.    The process $X_i$  is piecewise constant, and therefore so is $\lambda_0(X_i)$. Thus, for any value of $h$, we can easily compute the integrals so long as we have stored the  jump times of $X_i$ and the value of $\lambda_0(X_i)$ at each jump.
%

\subsection{Optimization problem}\label{sec:optimization_problem}
Given a reaction network, our goal is to find values of $n$, $m$, and $h$ that minimize the mean integrated squared error (MISE) \cref{def:MISE} for our conditional Monte Carlo estimator \eqref{def:conditional_mc_estimator} while staying within our computational budget of $b$. More precisely, we want to solve the following optimization problem
\begin{equation}
\label{eq:tominimize}
\begin{split}
&\min_{n,m,h} \,\,\underbrace{\mathbb{E}_{\nu,0}\left[\sum_{x\in\mathbb{Z}_{\ge 0}^d} \!\! \big(\hat{p}_t^\nu(x;n,m,h)-p_t^\nu(x)\big)^2\right]}_{\text{mean integrated squared error (MISE)}}, 
\end{split}
\end{equation}
 subject to
 \begin{equation}
 \begin{split}
&n\cdot  c\left(\mathbb{E}_{\nu,0}\left[\int_0^{t-h}\!\!\!\!\!\! \lambda_0(X(s))\,ds\right] + m \cdot \mathbb{E}_{\nu,0}\left[\int_{t-h}^t \!\!\!\! \lambda_0(X(s))\,ds\right]\right) \le b\\
&n,m\in\mathbb{Z}_{\ge 1} \text{ and } 0 \le h \le t.
\end{split}
\label{eq:constraint}
\end{equation}

The following theorem will allow us to transform the above optimization problem into a more solvable form.
\begin{theorem}\label{thm:MISE_simplification}
Suppose the process $X$ is non-explosive.  For any fixed $n,m\in\mathbb{Z}_{\ge 1}$ and $h\in[0,t]$
\begin{multline*}
\mathbb{E}_{\nu,0}\left[\sum_{x\in\mathbb{Z}_{\ge 0}^d} \!\! \big(\hat{p}_t^\nu(x;n,m,h)-p_t^\nu(x)\big)^2 \right] = \\
\frac{1}{n}\left[\frac{1}{m} + \left(1-\frac{1}{m}\right) P_\nu(X_{11}(t)=X_{12}(t))  - \sum_{x\in\mathbb{Z}_{\ge 0}^d} \!\!p_t^\nu(x)^2 \right].
\end{multline*}
\end{theorem}
The proof of \cref{thm:MISE_simplification} can be found in \cref{sec:proofThm31}.

If we allow $n$ to be continuous, then we can use the constraint \cref{eq:constraint} to solve for $n^{-1}$, and subsequently eliminate the constraint by substitution. This leads to a simpler 
optimization problem. In particular, let
\begin{align*}
	f(m,h) \stackrel{\text{def}}{=}  &\left(\frac{1}{m}\,\mathbb{E}_{\nu,0}\left[\int_0^{t-h}\!\!\!\!\!\! \lambda_0(X(s))\,ds\right] + \mathbb{E}_{\nu,0}\left[\int_{t-h}^t \!\!\!\! \lambda_0(X(s))\,ds\right]\right)\\
&\times \left(1 + (m-1) P_\nu(X_{11}(t)= X_{12}(t)) - m \!\! \sum_{x\in\mathbb{Z}_{\ge 0}^d} \!\! p_t^\nu(x)^2 \right).
\end{align*}
Then the original optimization problem \cref{eq:tominimize,eq:constraint} is equivalent to
\begin{equation}\label{def:bound_constrained_optimization_problem}
\begin{split}
&\min_{m,h} f(m,h)\\
&m\in\mathbb{Z}_{\ge 1}, 0 \le h \le t.
\end{split}
\end{equation}
Note that both $c$ and $b$ have dropped out of the optimization problem.

There are three terms in $f$ that we must know, or be able to approximate, in order to solve \cref{def:bound_constrained_optimization_problem}.
\begin{itemize}
\item The expectations of the integrals.  We discussed how to approximate these in \cref{sec:comp_cost_model}.
\item The sum $\sum_x p_t^\nu(x)^2$. However, we note that $\sum_x p_t^\nu(x)^2$ is the probability that two independent paths end up in the same state at time $t$. For many models, including the ones we tested, that sum is much smaller than $P_\nu(X_{11}(t)=X_{12}(t))$ and is close to zero. Thus for our examples, we replace the sum with zero and make that our general recommendation.  
\item The term $P_\nu(X_{11}(t)=X_{12}(t))$, whose approximation is the subject of the next section.
\end{itemize}

Note that there are many models for which $\sum_x p_t^\nu(x)^2$ will not be near zero.  However, for such models a small number of states will necessarily have a large probability.  An example of such a model would be a Birth-Death model, as in Section \ref{sec:examples}, with input rate 1 and output rate 1.  Such a model has a stationary distribution that is Poisson with a parameter of 1  \cite{anderson2010product}, and so for large $t$ the distribution $p_t^\nu$ will concentrate on the set $\{0,1,2,3\}$.  Other examples where $\sum_x p_t^\nu(x)^2$ is not small include those with extinction events.  For such models, it would not be appropriate to set this term to zero.  However, for models with diffuse probability mass functions, i.e.,  those models for which estimating $p_t^\nu$ is difficult and are the focus of this paper, the assumption will often be valid.

\subsection{Approximating the joint probability}
In order to optimize the objective function $f(m,h)$ in \cref{def:bound_constrained_optimization_problem}, we need to know, or be able to quickly approximate, the term $P_\nu(X_{11}(t)= X_{12}(t))$.   The following theorem, proven in \cref{sec:proofThm32}, will allow us to make a good approximation, without requiring any additional simulations.  The theorem makes use of the Skellam$(\mu_1,\mu_2)$ distribution, which is the distribution of the difference between two independent Poisson random variables with parameters $\mu_1$ and $\mu_2$, respectively.

\begin{theorem}\label{thm:skellam_approx}
	Let $S$ be the $d\times R$  matrix whose $r$th column is $\zeta_r$ and let  $\text{null}(S)$  be the right nullspace of $S$ restricted to integer values.  Let $X$ and $Z$ satisfy
	\begin{align*}
		X(t) &= X(0) + \sum_{r=1}^R Y_r^{X}\left( \int_0^t \lambda_r(X(s)) ds \right)\zeta_r,\\
		Z(t) &= X(0) + \sum_{r=1}^R Y_r^{Z}\left( \int_0^t \lambda_r(X(s)) ds \right)\zeta_r,
	\end{align*}
	where the $Y_r^{X}$ and $Y_r^{Z}$ are independent, unit-rate Poisson processes.  Assume that $X$ is non-explosive. For each $1 \le r \le R$ and $0\le a \le b \le t$, denote 
	\[
		\Lambda_r^{a,b} = \int_a^b \lambda_r(X(s)) ds,
	\]
	and let $K_r^{a,b}$ have the Skellam($\Lambda_r^{a,b},\Lambda_r^{a,b})$ distribution.  Then
\begin{equation}\label{eqn:joint_prob}
	P_\nu(X(t) = Z(t)) = \!\!\!\! \sum_{k\in\text{null}(S)} \!\!\!\! \mathbb{E}_{\nu,0}\left[\prod_{r=1}^R P\left(\left.K_r^{0,t} = k_r \,\right|\, \Lambda_r^{0,t}\right)\right].
\end{equation}
\end{theorem}
Note that $X$ is the process \cref{def:ctmc_random_time_change} that is of interest to us.  Returning to our setup, if we assume that
\begin{equation*}
\int_{t-h}^t \!\!\!\! \lambda_r(X_{11}(s))\,ds \approx \int_{t-h}^t \!\!\!\! \lambda_r(X_{12}(s))\,ds,
\end{equation*}
which should be valid for small $h$,
then \cref{thm:skellam_approx} leads to an approximation of $P_\nu(X_{11}(t)= X_{12}(t))$. In particular, we may sample $\tilde{n}$ paths and for the $i$th such path define
\begin{equation*}
\Lambda_{r,i}^{t-h,t} = \int_{t-h}^t \!\!\!\! \lambda_r(X_i(s))\,ds,\, 1\le i \le \tilde{n}.
\end{equation*}
Then $P_\nu(X_{11}(t)=X_{12}(t)) \approx \hat{P}_\nu(X_{11}(t)=X_{12}(t))$, where
\begin{equation}\label{def:joint_prob_approx}
\hat{P}_\nu(X_{11}(t)=X_{12}(t)) \stackrel{\text{def}}{=} \sum_{k\in\tilde{N}} \frac{1}{\tilde n}\sum_{i=1}^{\tilde{n}} \prod_{r=1}^R P\left(\left.K_r^{t-h,t} = k_r \,\right|\, \Lambda_{r,i}^{t-h,t}\right),
\end{equation}
and $\tilde{N}$ is a finite subset of $\text{null}(S)$.

To find $\tilde{N}$, we use the ``Algorithm for Solving the Linear Diophantine Equation Problem'' from section 1.5.2 of \cite{conforti2014integer}. In general, the algorithm finds solutions $x\in\mathbb{Z}^d$ to linear equations of the form $Ax=b$ for rational $A$ and $b$. In our case, we enumerate solutions to $S k=0$ for $k\in\mathbb{Z}^d$. Generally, there are infinitely many solutions, however the right-hand side of \eqref{def:joint_prob_approx} is always maximized at $k=0$, and decreases as $k$ moves away from $0$. Thus we approximate \eqref{def:joint_prob_approx} by starting at $k=0$ and enumerating all ``nearby'' solutions. \Cref{alg:enumerate_null_S} shows how to apply the algorithm from \cite{conforti2014integer} to our particular problem.  In all of our numerical examples, we chose $C = 4$ in \Cref{alg:enumerate_null_S}.

\begin{algorithm}
    \caption{Algorithm for enumerating a finite subset of $\text{null}(S) \in\mathbb{Z}^{d \times R}$}
    \begin{algorithmic}[1] 
    	\Require the stoichiometry matrix $S$ and $C \in\mathbb{Z}_{> 0}$
       	\If {$S$ does not have full row rank}
       	    \State Remove redundant equations from the system and replace $S$.
        \EndIf

        \State
        \State Transform $S$ into its Hermite normal form $H$, and store the matrix $U$ that satisfies $H = SU$.
        
        \State $r \gets R - \text{rank}(S)$
        \State Let $\tilde{U}$ be the matrix containing the last $r$ columns of $U$.
        
        \State
        \State $\tilde{N} \gets \left\{\left.\tilde{U} z \right\vert z\in \mathbb{Z}^{r}, ||z||_\infty \le C \right\}$

    \end{algorithmic}\label{alg:enumerate_null_S}
\end{algorithm}

\begin{algorithm}
    \caption{Algorithm for computing $\hat{P}_\nu(X_{11}(t)=X_{12}(t))$}
    \begin{algorithmic}[1] 
    	\Require $\tilde{n}$ i.i.d.\ samples of $X$, denoted $\{X_i\}_{i=1}^{\tilde{n}}$ \Comment $\tilde{n} = 500$ was more than sufficient.
    	\Require the stoichiometry matrix $S$, and a finite $\tilde{N}\subset \text{null}(S)$
    	
    	\ForAll{$r$ in $1,\ldots,R$ and $i$ in $1, \ldots, \tilde{n}$ }
    		\State $\Lambda_{r,i}^{t-h,t} \gets \int_{t-h}^t \lambda_r(X_i(s))\,ds$
    	\EndFor 	
		\State
       	\State $\hat{P} \gets 0$
       	\ForAll{$k$ in $\tilde{N}$ and $i$ in $1, \ldots, \tilde{n}$}
    		\State $\hat{P} \gets \hat{P} + \prod_{r=1}^R P\left(\left.K_r^{t-h,t} = k_r \,\right|\, \Lambda_{r,i}^{t-h,t}\right)$
    		\Comment $K_r^{t-h,t} \sim \text{Skellam}(\Lambda_{r,i}^{t-h,t}, \Lambda_{r,i}^{t-h,t})$
    	\EndFor
       	
       	\State
       	\State $\hat{P}_\nu(X_{11}(t)=X_{12}(t)) \gets \hat{P}/\tilde n $
    \end{algorithmic}\label{alg:joint_probability_estimation}
\end{algorithm}

\subsection{Approximation to the optimization problem}
By using the joint probability approximation \cref{def:joint_prob_approx}, we can approximate the  function $f$ in the optimization problem \cref{def:bound_constrained_optimization_problem}.  In particular, let
\begin{align}
\label{def:f_hat}
\begin{split}
\hat{f}(m,h) \stackrel{\text{def}}{=}  &\left(\frac{1}{m}\int_0^{t-h}\!\!\!\!\!\! \bar{\lambda}_0(X(s))\,ds + \int_{t-h}^t \!\!\!\! \bar{\lambda}_0(X(s))\,ds\right)\\
&\left(1 + (m-1) \hat{P}_\nu(X_{11}(t)= X_{12}(t))- m \cancelto{0}{\sum_{x\in\mathbb{Z}_{\ge 0}^d} \!\! p_t^\nu(x)^2}\,\,\,\,\,\,\right),
\end{split}
\end{align}
where $\bar{\lambda}_0(X(s)) = \frac{1}{\tilde{n}}\sum_{i=1}^{\tilde{n}} \sum_{r=1}^R \lambda_r(X_i(s))$, and the $\{X_i\}_{i=1}^{\tilde{n}}$ are independent paths of $X$. Then we may substitute $f$ with $\hat f$ and our new optimization problem is the following:	
\begin{equation}\label{def:approx_unconstrained_optimization_problem}
\begin{split}
&\min_{m,h} \hat{f}(m,h)\\
&m\in\mathbb{R}_{\ge 1}, 0 \le h \le t.
\end{split}
\end{equation}
Note that above we have allowed $m$ to be real--valued, as opposed to integer valued.  This allows us to use continuous optimization algorithms, which generally converge more rapidly. According to Figure SM1, which shows $\hat{f}(m,h)$ for many values of $m$ and $h$, $\hat{f}$ does not change too quickly with $m$, so allowing $m$ to range over the reals instead of the integers should not change the optimal values of $m$ and $h$ appreciably.  

It is important to know when the optimization problem \cref{def:approx_unconstrained_optimization_problem} has a finite solution.  In the proposition below, we show that a solution necessarily exists when $\hat{P}_\nu(X_{11}(t)=X_{12}(t))$ is larger than the approximation used for $\sum_x p_t^\nu(x)^2$.  Since we approximate the sum with zero, we may conclude that a finite solution always exists in our setup.
\begin{proposition}\label{thm:finite_solution_to_optimization_problem}
Let $\widehat{p^2}$ be our approximation to $\sum_x p_t^{\nu}(x)^2$. If $\hat{P}_\nu(X_{11}(t)=X_{12}(t)) > \widehat{p^2}$ for all $h\in[0,t]$, then \cref{def:approx_unconstrained_optimization_problem} has a finite solution.
\begin{proof}
Since the integrals are nonnegative, $h$ is in a compact domain, $\hat f$ depends continuously on $h$ and $m$, and $\lim_{m\to\infty} \hat{f}(m,h) = \infty$, a finite solution exists. 
\end{proof}
\end{proposition}

\Cref{alg:cond_mc} outlines the full conditional Monte Carlo algorithm, which brings together all of the individual pieces of the algorithm that we previously discussed.

\begin{algorithm}
    \caption{Conditional Monte Carlo algorithm}
    \begin{algorithmic}[1] 
    	\Require $\tilde{n}$ i.i.d.\ samples of $X$, denoted $\{X_i\}_{i=1}^{\tilde{n}}$ \Comment $\tilde{n} = 500$ was more than sufficient.
  
       	\State $m,h \gets \argmin_{\substack{m\in\mathbb{R}_{\ge 1}\\0\le h\le t}} \hat{f}(m,h)$
       	\State
       	\Comment Use  $\{X_i\}_{i=1}^{\tilde{n}}$, \cref{def:f_hat}, and \cref{alg:joint_probability_estimation} to evaluate $\hat{f}$.
       	
       	\ForAll{ $i$ in $1, \ldots, n$}
    		\State Sample $X_i(t-h)$.
    		\Comment The $X_i(t-h)$ are i.i.d.
    		
    		\ForAll{ $j$ in $1, \ldots, m$}
    			\State Sample $X_{ij}(t)$ conditioned on $X_{ij}(t-h)=X_i(t-h)$.
    			\State
    			\Comment See \cref{sec:intro} for details about $X_{ij}$.
    		\EndFor
    	\EndFor
       	
       	\State
       	\State $\hat{p}_t^\nu(x;n,m,h) \gets \frac{1}{n} \sum_{i=1}^n \frac{1}{m} \sum_{j=1}^m \mathbbm{1}(X_{ij}(t)=x)$
       	\end{algorithmic}\label{alg:cond_mc}
\end{algorithm}

\section{Numerical results}\label{sec:numerical_results}


In this section, we present numerical results demonstrating the improvement in accuracy, quantified via the mean integrated squared error \eqref{def:MISE},  that comes from using our conditional Monte Carlo estimator instead of the classical Monte Carlo estimator. In particular, when near--optimal values of $m$ and $h$ are utilized, the accuracy often improves by an order of magnitude for a fixed computational budget. Moreover, we show that  the function $\hat f$ of \cref{def:approx_unconstrained_optimization_problem} is indeed a very good approximation for $f$ of \cref{def:bound_constrained_optimization_problem} for the examples we considered,  allowing us to conclude that the values of $m$ and $h$ our method produces are near--optimal.

The following steps were carried out on each of our test examples.  First, we fixed an integer $n_1$ and computed the classical Monte Carlo estimator
\begin{equation*}
p_t^{\text{MC}}(x;n_1) = \frac{1}{n_1}\sum_{i=1}^{n_1} \mathbbm{1}(X_i(t)=x),\, x \in \Z^d_{\ge 0}.
\end{equation*}
For all models, we used $n_1=10^4$. We also recorded the number of random variates used in generating $p_t^{\text{MC}}(\ \cdot\ ;n_1)$, which served as the budget $b$ in the computational cost constraint \cref{eq:constraint}. 

After obtaining $p_t^{\text{MC}}(\ \cdot\ ;n_1)$,  we computed the conditional Monte Carlo estimator
\begin{equation*}
p_t^{\text{CMC}}(x;n_2,m,h) = \frac{1}{n_2} \sum_{i=1}^{n_2} \frac{1}{m} \sum_{j=1}^m \mathbbm{1}(X_{ij}(t)=x),\, x \in \Z^d_{\ge 0},
\end{equation*}
for various pairs of $m$ and $h$, and $n_2$ was allowed to increase until the  conditional estimator used essentially the same number of random variates as the classical Monte Carlo estimator. All random variates generated for the conditional estimators were independent of those utilized for the classical estimator.

Next, for both classical and conditional Monte Carlo, we computed the integrated squared error
\begin{equation}\label{def:ise}
\text{ISE} = \sum_{\tilde{S}} \left(\hat{p}(x) - p_t^\nu(x) \right)^2,
\end{equation}
where $\tilde{S}$ was a large fixed subset of the state space, and $\hat{p}(x)$ was either the classical or conditional Monte Carlo estimate. The ISE is itself a random variable, and so we approximated the mean integrated square error (MISE) by averaging 100 independent samples of the ISE.

The exact values of $p_t^\nu(x)$ were unknown. Thus the values were estimated with conditional Monte Carlo with a large value of $n_1$ (we used $n_1 = 10^9$), and with $m$ and $h$ chosen so that they approximately minimize the MISE. 

\begin{figure}
\centering
\includegraphics[width=0.7\textwidth]{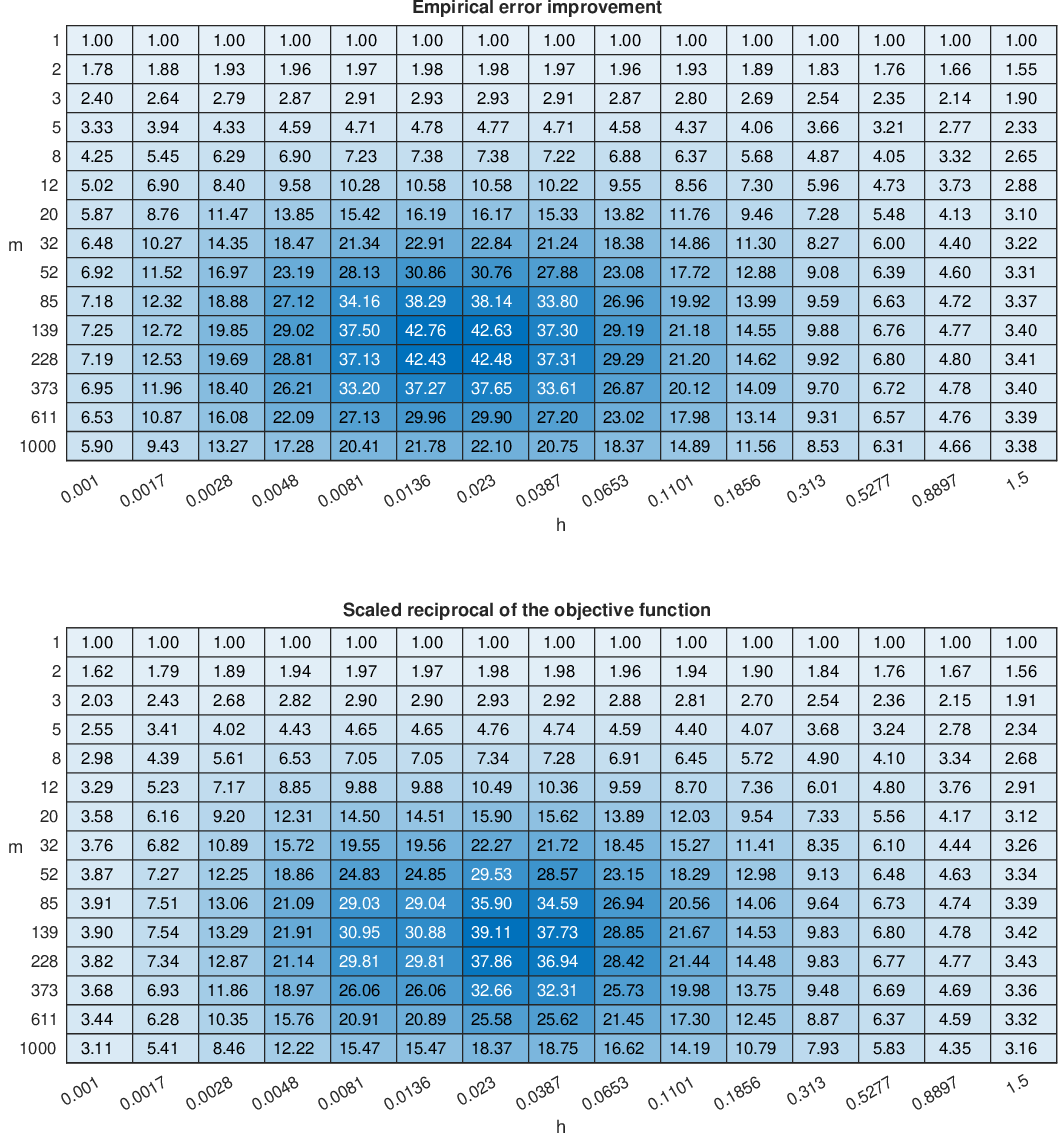}
\caption{Lotka-Volterra model. The first heatmap shows $\text{MISE}_\text{MC} / \text{MISE}_\text{CMC}(m,h)$ for different values of $m$ and $h$. The method we used to obtain the ratio is described in \cref{sec:numerical_results}. The second heatmap shows that value of $\hat{f}(1,0)/\hat{f}(m,h)$. The definition of $\hat{f}$ is given by \cref{def:f_hat}.} 
\label{fig:heatmap_lotka_volterra}
\end{figure}

\begin{figure}
\centering
\includegraphics[width=0.7\textwidth]{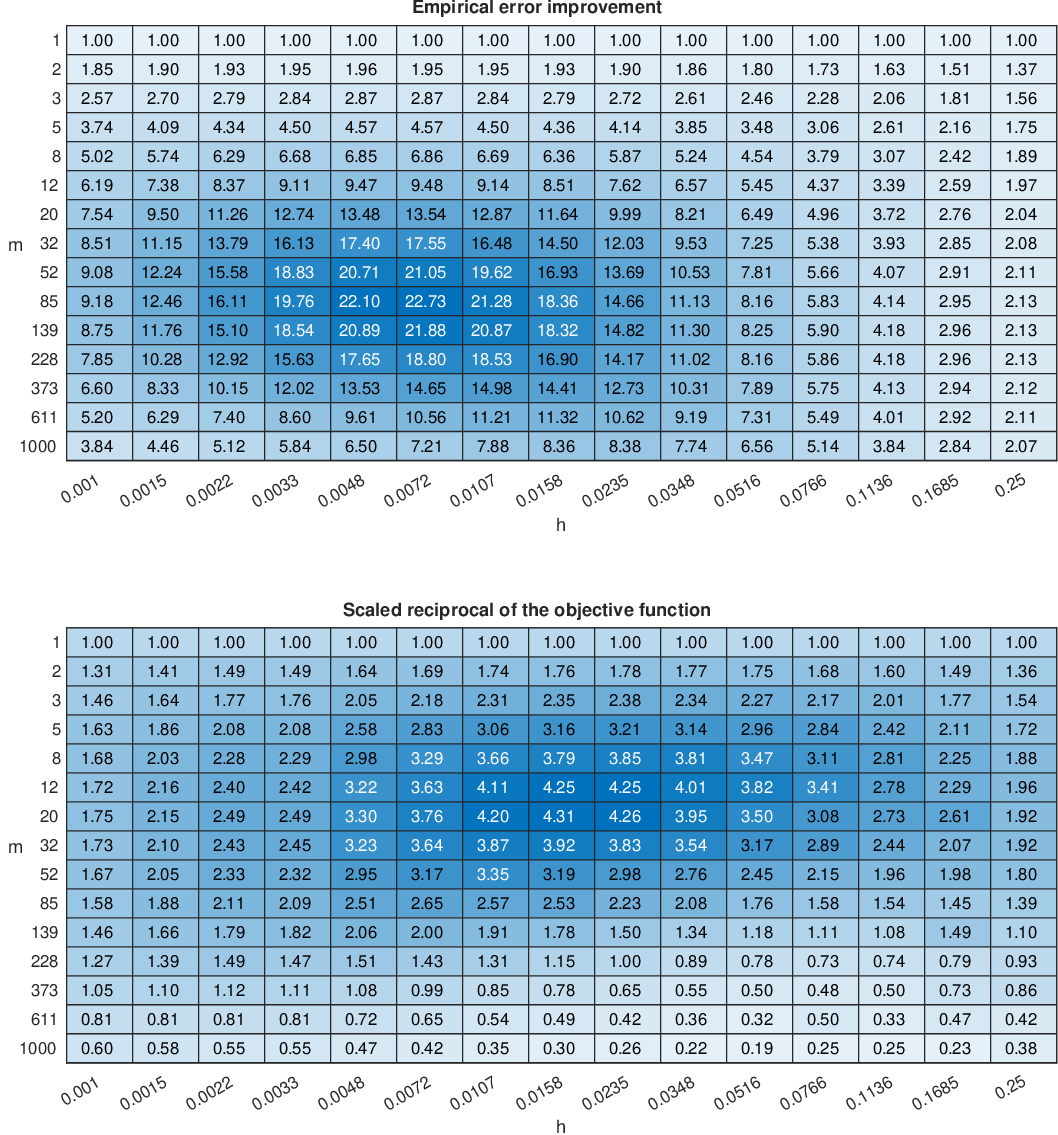}
\caption{Dimerization model. The first heatmap shows $\text{MISE}_\text{MC} / \text{MISE}_\text{CMC}(m,h)$ for different values of $m$ and $h$. The method we used to obtain the ratio is described in \cref{sec:numerical_results}. The second heatmap shows that value of $\hat{f}(1,0)/\hat{f}(m,h)$. The definition of $\hat{f}$ is given by \cref{def:f_hat}.}
\label{fig:heatmap_dimerization}
\end{figure}
Finally, we denote by $\text{MISE}_\text{MC}$  our estimate of the classical Monte Carlo MISE, and, for a given $m$ and $h$, we denote by  $\text{MISE}_\text{CMC}(m,h)$  the conditional version. For each model, and for each choice of $m$ and $h$, an ``empirical error improvement'' was computed as the following ratio   
\begin{equation*}
\frac{\text{MISE}_\text{MC}}{\text{MISE}_\text{CMC}(m,h)},
\end{equation*}
where a  number greater than one implies that conditional Monte Carlo has a lower MISE than classical Monte Carlo when given the same computational budget.
These values, one for each pair of $m$ and $h$,  can then be plotted.  
In the top half of \Cref{fig:heatmap_lotka_volterra,fig:heatmap_dimerization} (and Figures SM2 to SM5), we display these values with a heatmap.  Of particular interest is the order of magnitude improvement in computational efficiency we see with the conditional Monte Carlo estimator as compared to classical Monte Carlo \textit{when well--chosen values of $h$ and $m$ are utilized}.  In particular, for the Lotka-Volterra model we see a 40-fold improvement, for the dimerization model we see a 20-fold improvement, for the toggle model we see a 20-fold improvement, and for the fast/slow model we see a 20-fold improvement. For the birth and birth--death models we see more modest improvements in computational efficiency, but this can be explained by the simplicity of these models which makes classical Monte Carlo sufficient for the task at hand. In particular, one promising aspect of the present work comes into focus with these numerical results:  the more complicated the model, and the larger and more diffuse the distribution of the model (which is where other methods, including those that approximately solve the chemical master equation directly, struggle), the better the performance of the conditional Monte Carlo estimator.

In practice, we are not given the optimal values of the parameters $m$ and $h$, so we find them via the optimization problem \cref{def:approx_unconstrained_optimization_problem}.  In each of the bottom portions of \Cref{fig:heatmap_lotka_volterra,fig:heatmap_dimerization} (and Figures SM2 to SM5), we provide the values of $\hat f(m,h)$ for the different pairs of $m$ and $h$.  We report the inverse so that the heatmap will agree qualitatively with the top portion of the figures (higher values are desirable).  We also normalized the values by multiplying them by $\hat{f}(1,0)$, which does not affect the results of the optimization problem in any way. To generate each value  $1/\hat{f}(m,h)$ we first sampled $\tilde{n}=500$ paths, which then allowed us to compute $\bar{\lambda}_0$ and $\hat{P}_\nu(X_{11}(t)=X_{12}(t))$ as detailed in the previous section. We could then use these values to compute $\hat{f}(m,h)$ via \cref{def:f_hat}.

Note that the empirical error improvement and $\hat{f}$ do not need to have the same value for a pair of $m$ and $h$. The important thing is that the maximizer of the empirical error improvement is similar to the minimizer of $\hat{f}$. The heatmaps do indeed suggest that the true and approximate optimization problems have similar solutions. 
What is also clear from these numerical results is that  even if $m$ and $h$ slightly deviate from their optimal values, we still get a substantial improvement.  

We stress that such heatmaps do not need to be made by anyone who uses the conditional Monte Carlo algorithm. They are only used here to demonstrate that the optimization problem \cref{def:approx_unconstrained_optimization_problem} can be safely used to find the near--optimal values of $m$ and $h$, which can then be used to construct the desired estimator \cref{def:conditional_mc_estimator} via \cref{alg:cond_mc}.

\section{A central limit theorem}\label{sec:clt} In this section,  we will show how to obtain an approximate one-sided confidence interval for the integrated squared error \cref{def:ise} without running more simulations. Specifically, for a fixed (presumably large) finite subset of the state space $\tilde S$, a fixed $\alpha\in(0,1)$, and large $n$, we want to find a sequence of positive constants $\{C_n\}$ and a constant $u > 0$ such that
\begin{equation}\label{eqn:ise_confidence_interval}
\lim_{n\to\infty} P\Big(C_n \underbrace{\sum_{x\in\tilde{S}} \big(\hat{p}_t^\nu(x;n,m,h)-p_t^\nu(x)\big)^2}_{\text{integrated squared error}} \le u\Big) = 1-\alpha,
\end{equation}
where $C_n$ is allowed to depend on $m$ and $h$. The following central limit theorem will lead us to values for $\{C_n\}$ and $u$.

\begin{theorem}\label{thm:clt_integrated_mse}
Fix $m\in\Z_{\ge 1}$ and $h\in[0,t]$. Let $\mathcal{S}\subset \mathbb{Z}_{\ge 0}^d$ be the state space of the continuous time Markov chain, and let $\tilde{\mathcal{S}}$ be a finite subset of $\mathcal{S}$. Choose an enumeration of $\tilde{\mathcal{S}}$ and denote it $\{x_i\}_{i=1}^{|\tilde{\mathcal{S}}|}$. Let $p_t^\nu,\hat{p}_t^\nu \in\mathbb{R}^{|\tilde{\mathcal{S}}|}$ with  their $i$th elements equal to $p_t^\nu(x_i)$ and $\hat{p}_t^\nu(x_i;n,m,h)$, respectively. Let
\begin{equation}\label{def:sigma}
\Sigma \stackrel{\text{def}}{=} m\,\text{diag}(p_t^\nu) + m(m-1) A - m^2 p_t ^\nu (p_t^\nu)^T,
\end{equation}
where $\text{diag}(p_t^\nu)$ is the diagonal matrix with $p_t^\nu$ along its diagonal, and $A$ is a $|\tilde{\mathcal{S}}| \times |\tilde{\mathcal{S}}|$ matrix where $A_{ij} = P_\nu(X_{11}(t)=x_i,X_{12}(t)=x_j)$. Then
\vspace{-0.1in}
\begin{equation}\label{eqn:clt_integrated_mse}
nm^2 \sum_{x\in\tilde{\mathcal{S}}} \big(\hat{p}_t^\nu(x;n,m,h)-p_t^\nu(x)\big)^2 \stackrel{d}{\to} \sum_{\ell=1}^{|\tilde{S}|} \lambda_\ell Z_\ell^2 \text{, as } n\to\infty,
\end{equation}
where  $\{\lambda_\ell\}_{\ell=1}^{|\tilde{S}|}$ are the eigenvalues of $\Sigma$ and $Z_\ell\stackrel{\text{i.i.d.}}\sim N(0,1)$.
\end{theorem}

$\Sigma$ is usually an enormous matrix, so we do not want to store it, much less compute its eigenvalues. The Satterthwaite approximation \cite{satterthwaite1941} says that
\begin{equation}\label{eqn:satterthwaite_approx}
\sum_\ell \lambda_\ell Z_\ell^2 \stackrel{d}\approx \frac{\sum_\ell \lambda_\ell^2}{\sum_{\ell}\lambda_\ell} \chi^2 \left(\frac{\left(\sum_{\ell}\lambda_\ell \right)^2}{\sum_{\ell} \lambda_\ell^2} \right) = \frac{\tr{\Sigma^2}}{\tr{\Sigma}} \chi^2 \left(\frac{\tr{\Sigma}^2}{\tr{\Sigma^2}} \right),
\end{equation}
where $\chi^2(v)$ denotes a $\chi^2$ random variable with $v$ degrees of freedom. The approximation is obtained by matching the first two moments of the linear combination (above left-hand side) and the chi-squared distribution (above right-hand side). The advantage of the approximation is that we can estimate $\tr{\Sigma}$ and $\tr{\Sigma^2}$ without storing $\Sigma$ explicitly or computing its eigenvalues.

\begin{theorem}\label{thm:traces}
Fix $n,m\in\mathbb{Z}_{\ge 1}$ and $h\in[0,t]$. Let $\tilde{\mathcal{S}}$, $\{x_k\}_{k=1}^{|\tilde{\mathcal{S}}|}$, and $\hat{p}_t^\nu$ be defined as in \cref{thm:clt_integrated_mse}. For $1 \le i \le n$, let $M_i\in\Z_{\ge 0}^{|\tilde{\mathcal{S}}|}$, and set its $k$th element to $M_i(x_k) \stackrel{\text{def}}{=} \sum_{j=1}^m \mathbbm{1}(X_{ij}=x_k)$ (the $\{X_{ij}\}$ are defined in \cref{sec:intro}). Let $\hat{\Sigma}_n$ be the usual sample covariance matrix of $\{M_i\}_{i=1}^n$. Specifically,
\begin{equation*}
\hat{\Sigma}_n \stackrel{\text{def}}{=} \frac{1}{n-1} \sum_{i=1}^n \left(M_i - \overline{M}\right) \left(M_i - \overline{M}\right)^T,
\end{equation*}
where $\overline{M} = n^{-1}\sum_{i=1}^n M_i$. Then
\begin{equation}\label{eqn:tr_sample_cov}
\tr{\hat{\Sigma}_n} = \frac{1}{n-1} \sum_{i=1}^n M_i^T M_i - \frac{nm^2}{n-1} (\hat{p}_t^\nu)^T \hat{p}_t^\nu,
\end{equation}
and
\begin{multline}\label{eqn:tr_sample_cov_sq}
\tr{\hat{\Sigma}_n^2} = \frac{1}{(n-1)^2} \sum_{i=1}^n \left[M_i^T M_i - 2\overline{M}^T M_i  + m^2 (\hat{p}_t^\nu)^T \hat{p}_t^\nu \right]^2 \\
+ \frac{2}{(n-1)^2} \sum_{1\le i < j\le n}\left[M_i^T M_j - \overline{M}^T M_i - \overline{M}^T M_j  + m^2 (\hat{p}_t^\nu)^T \hat{p}_t ^\nu \right]^2.
\end{multline}
Furthermore
\begin{equation*}
\tr{\hat{\Sigma}_n} \stackrel{\text{a.s.}}{\to} \tr{\Sigma} \text{ and } \tr{\hat{\Sigma}_n^2} \stackrel{\text{a.s.}}{\to} \tr{\Sigma^2} \text{ as } n\to\infty.
\end{equation*}
\end{theorem}

For the models we tested, the optimal value of $m$ was only moderately large (on the order of 10 to 100), and the indicator in the summand of $M_i(x)$ is zero for many values of $x$. Whenever those two conditions hold, $M_i$ sparse. Consequently, storing $\{M_i\}_{i=1}^n$ does not require too much memory, and the terms $M_i^T M_j$ and $\overline{M}^T M_i$ are  cheap to compute. \Cref{alg:trace_estimation} summarizes how we compute the traces. Using the sparsity of the $M_i$ is important, because otherwise the vectors are too large to store and the operations are slow.

\begin{algorithm}
    \caption{Algorithm for computing $\hat{p}_t^\nu, \tr{\hat{\Sigma}_n}$, and $\tr{\hat{\Sigma}^2_n}$ }
    \begin{algorithmic}[1] 
    	\Require $n,m\in\mathbb{Z}_{\ge 1}$ and $h\in[0,t]$
		\For{$i$ in $\{1,\ldots, n\}$}
			\State Sample $X_i(t-h)$.
			\State Given $X_i(t-h)$, sample $\{X_{ij}(t)\}_{j=1}^m$.
			\For{$x$ in $\tilde{S}$}
				\State $M_i(x) \gets \sum_{j=1}^m \mathbbm{1}(X_{ij}(t)=x)$ \Comment{Store $M_i$ as a sparse vector.}
			\EndFor
		\EndFor
		
		\State
		\State $\hat{p}_t^\nu \gets \frac{1}{nm} \sum_{i=1}^n M_i$
		\State Compute $\tr{\hat{\Sigma}_n}$ according to \cref{eqn:tr_sample_cov}.
		\State Compute $\tr{\hat{\Sigma}^2 _n}$ according to \cref{eqn:tr_sample_cov_sq}.
    \end{algorithmic}\label{alg:trace_estimation}
\end{algorithm}

\begin{corollary}\label{thm:upper_confidence_bound}
Fix $n,m\in\mathbb{Z}_{\ge 1}$ and $h\in[0,t]$. Also fix an $\alpha\in(0,1)$, and let $\chi_\alpha^2(v)$ be the $1-\alpha$ quantile of the $\chi^2$ distribution with $v$ degrees of freedom. An approximate $1-\alpha$ confidence interval for $\sum_{x\in\tilde{\mathcal{S}}} \big(\hat{p}_t^\nu(x;n,m,h)-p_t^\nu(x)\big)^2$ is $[0,U_n/(nm^2)]$, where
\begin{equation}\label{def:upper_confidence_bound}
U_n \stackrel{\text{def}}{=} \frac{\tr{\hat{\Sigma}_n^2}}{\tr{\hat{\Sigma}_n}}\chi_\alpha^2\left(\frac{\tr{\hat{\Sigma}_n}^2}{\tr{\hat{\Sigma}_n^2}} \right).
\end{equation}
\end{corollary}

\Cref{fig:clt_lotka_volterra,fig:clt_dimerization} (and also Figures SM6 to SM9), compare the empirical distribution of
\begin{equation}\label{def:normalized_ise}
nm^2\sum_{x\in\tilde{\mathcal{S}}} \big(\hat{p}_t^\nu(x;n,m,h)-p_t^\nu(x)\big)^2
\end{equation}
to the approximate asymptotic distribution \cref{eqn:satterthwaite_approx}, where the true traces are replaced with the sample traces from \cref{alg:trace_estimation}. The figures also compare the sample 95\% quantile to the same quantile based on \cref{thm:upper_confidence_bound}, which turned out to be close.

\begin{figure}\label{fig:clt}
	\centering
    \subfloat[Lotka-Volterra model.\label{fig:clt_lotka_volterra}]{{\includegraphics[width=0.6\textwidth]{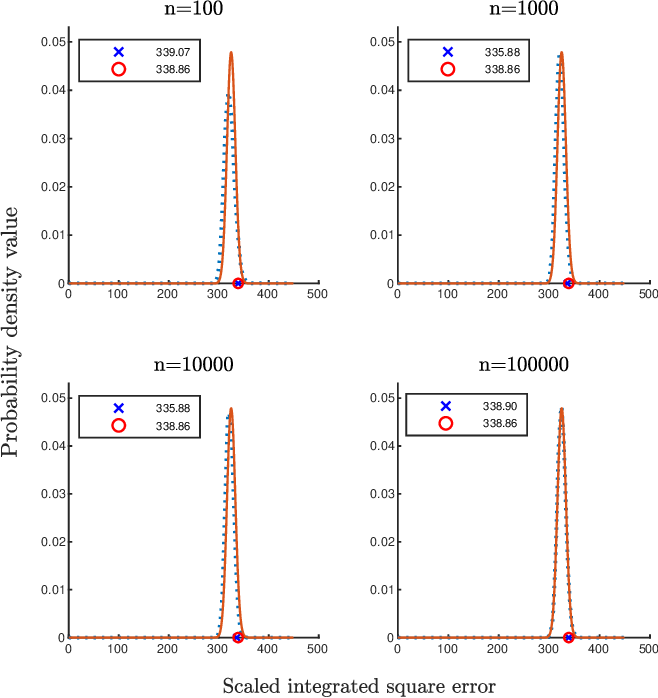}}}	
    \qquad
    \subfloat[Dimerization model.\label{fig:clt_dimerization}]{{\includegraphics[width=0.6\textwidth]{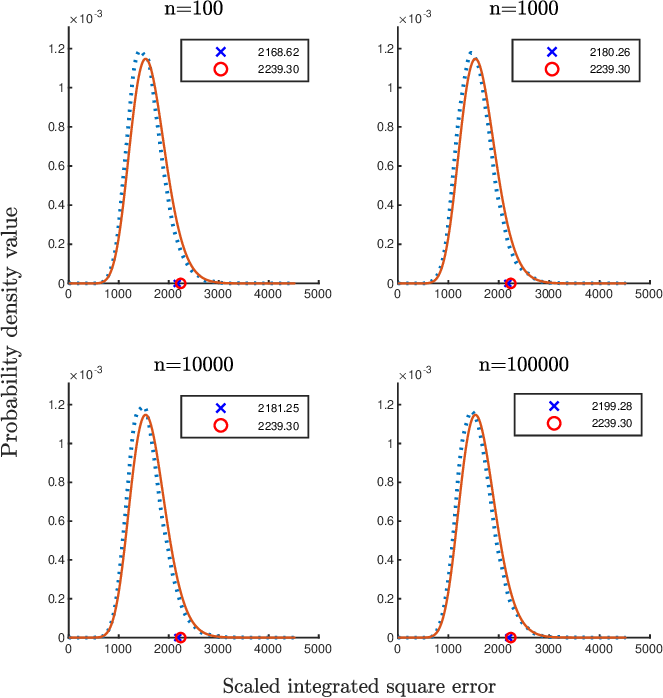}}}
    \caption{The dashed blue density is the empirical density of the integrated squared error \cref{def:normalized_ise}, whereas the solid red density is the Satterwaithe approximation to the asymptotic density \cref{eqn:satterthwaite_approx}. The blue cross and red circle are the 95\% quantiles of their respective densities. To generate the blue curve, first we sampled $10^4$ values of $nm^2\sum_{x\in\tilde{\mathcal{S}}} \big(\hat{p}_t^\nu(x;n,m,h)-p_t^\nu(x)\big)^2$ (which we call the ``scaled integrated squared error'') for different values of $n$. Given those samples, we used MATLAB's \textit{ksdensity} function to generate the blue curve. The traces of $\Sigma$ and $\Sigma^2$ were estimated with an independent set of $10^5$ simulations and \cref{alg:trace_estimation}.}
\end{figure}

\section{Directions for future research}\label{sec:conclusion}
We demonstrated how to implement a version of conditional Monte Carlo in the context of continuous time Markov chain models for reaction networks. There are many possible directions for future research; we list three.
\begin{enumerate}
\item The method could be extended so it provides estimates of the distribution at multiple fixed time-points. The method we developed, and in particular the optimization problem we utilize to find the values of $m$ and $h$, is tailored to the single time-point case.


\item In the method developed here the conditional expectation in \cref{def:estimator_with_cond}
\[
	\mathbb{E}_{X_i(t-h),t-h}\left[\mathbbm{1}(X(t)=x) \right]
\]
is approximated by Monte Carlo with $m$ conditionally independent realizations. However, it could be approximated by solving the chemical master equation directly, perhaps via the finite state projection algorithm \cite{munsky2006finite}. Because the solver need only integrate the system of ODEs over the time interval $[t-h,t]$, the probability mass should not become too diffuse, thereby solving one of the major difficulties related to these solvers.
 
We implemented this approach and observed some increase in efficiency over the conditional Monte Carlo algorithm \cref{alg:cond_mc}, around a factor of three.  However, the gains were only realized when an optimal value of $h$ was chosen, and we needed to test many different $h$ values in order to find  the optimal value. In practice, we would need a faster method for finding the optimal parameters, similar to the optimization problem detailed in this paper.

\item As discussed in the introduction and  \cref{sec:expectations}, the present method is not optimized for the estimation of expectations.  Developing a new conditional Monte Carlo estimator tailored to that problem is a natural focus of future work.
\end{enumerate}

%

%
%
%
%
%

\appendix
\section{Proofs}\label{sec:additional_proofs}

\subsection{Theorem regarding the expected number of reactions}
\begin{theorem}\label{thm:expected_number_reactions}
Suppose that the process $X$ is non-explosive and fix $h\in[0,t]$ and $m\in\Z_{\ge 1}$.  Then the expected number of reactions required to sample $\{X_{1j}\}_{j=1}^m$ is
\begin{equation*}
\mathbb{E}_{\nu,0}\left[\int_0^{t-h}\!\!\!\!\!\! \lambda_0(X(s))\,ds\right] + m\, \mathbb{E}_{\nu,0}\left[\int_{t-h}^t \!\!\!\! \lambda_0(X(s))\,ds\right].
\end{equation*}
\end{theorem}
\begin{proof}
The number of reactions required to sample $\{X(s)\}_{s\in[a,b]}$ is
\begin{equation*}
\sum_{r=1}^R \left[Y_r\left(\int_0^b \lambda_r\left(X(s)\right)\,ds\right)- Y_r\left(\int_0^a \lambda_r\left(X(s)\right)\,ds\right)\right],
\end{equation*}
where the $Y_r$ are independent unit-rate Poisson processes \cite{kurtz1980representations}. For each $r$,
\begin{equation*}
Y_r\left(\int_0^t \lambda_r\left(X(s)\right)\,ds\right) - \int_0^t \lambda_r\left(X(s)\right)\,ds
\end{equation*}
is a martingale \cite[Theorem 1.22]{AKbook}, so the result follows.
\end{proof}

\subsection{Proof of \texorpdfstring{\cref{thm:MISE_simplification}}{MISE simplification}}
\label{sec:proofThm31}

For simplicity, denote $X_{ij}(t)$ as $X_{ij}$. We start with the left-hand side of the desired equality. The monotone convergence theorem implies that we can move the expectation inside the sum, by which we mean
\begin{align*}
\mathbb{E}_{\nu,0}\left[\sum_x \big(\hat{p}_t^\nu(x;n,m,h)-p_t^\nu(x)\big)^2 \right] &= \sum_x \mathbb{E}_{\nu,0}\left[ \big(\hat{p}_t^\nu(x;n,m,h)-p_t^\nu(x)\big)^2 \right]\\
&= \sum_x \text{Var}[\hat{p}_t^\nu(x;n,m,h)].
\end{align*}
The last line follows from the fact that the estimator $\hat{p}_t^{\nu}$ is unbiased. From the definition of $\hat{p}_t^{\nu}$, and also basic properties of variance, the above is equal to
\begin{align*}
&= \sum_x \text{Var}\left[\frac{1}{nm}\sum_{i=1}^n\sum_{j=1}^m \mathbbm{1}(X_{ij}=x) \right]\\
&= \frac{1}{nm^2}  \sum_x \left[\sum_{j=1}^m \text{Var}[\mathbbm{1}(X_{1j}=x)]  + 2 \!\!\!\! \sum_{1\le i < j \le m} \!\!\!\! \text{Cov}\big(\mathbbm{1}(X_{1i}=x), \mathbbm{1}(X_{1j}=x)\big)\right]\\
&= \frac{1}{nm^2}  \sum_x \left[m \text{Var}[\mathbbm{1}(X_{11}=x)] + m(m-1) \text{Cov}(\mathbbm{1}(X_{11}=x), \mathbbm{1}(X_{12}=x))\right]\\
&= \frac{1}{nm}\sum_x \left[p_t^\nu(x)(1\!-\!p_t^\nu(x))+ (m\!-\!1) \Big(\mathbb{E}_{\nu,0}\left[\mathbbm{1}(X_{11}=x)\mathbbm{1}(X_{12}=x)\right] - p_t^\nu(x)^2\Big) \right]\\
&= \frac{1}{nm}\sum_x \left[p_t^\nu(x)+ (m-1) P_\nu(X_{11}=x,X_{12}=x) - mp_t^\nu(x)^2 \right]\\\
&= \frac{1}{n}\left[\frac{1}{m} + \left(1-\frac{1}{m}\right) P_\nu(X_{11}=X_{12})  - \sum_x p_t^\nu(x)^2 \right].
\end{align*}
We can also take $p_t^\nu(x)$ to be a marginal distribution. In that case, interpret sums over $x$ as sums over the lower-dimensional marginal variables. Also, view $X_{11}=X_{12}$ as being true if their coordinates corresponding to the marginal variables are equal.

\subsection{Proof of \texorpdfstring{\cref{thm:skellam_approx}}{Skellam approximation of the joint probability}}
\label{sec:proofThm32}
Let $\Lambda^{0,t}\in\mathbb{R}_{\ge 0}^R$ be the vector whose $r$th element is $\Lambda_r^{0,t}$, and let $Y^X,Y^Z\in\mathbb{Z}_{\ge 0}^R$ be the vectors whose $r$th elements are $Y_r^X(\Lambda_r^{0,t})$ and $Y_r^Z(\Lambda_r^{0,t})$, respectively. Then
\begin{align*}
P_\nu(X(t) = Z(t)) &= P_\nu\left(S Y^X = S Y^Z\right)\\
&=  P_\nu\left(S(Y^X - Y^Z) =0 \right)\\
&= \!\!\!\! \sum_{k\in\text{null}(S)} \!\!\!\!\! P_\nu(Y^X - Y^Z = k)\\
&= \!\!\!\! \sum_{k\in\text{null}(S)} \!\!\!\!\! \mathbb{E}_{\nu,0}\left[P(\left.Y^X-Y^Z = k\,\right|\, \Lambda^{0,t})\right].
\end{align*}
The elements of $Y^X$ and $Y^Z$ are independent when conditioned on $\Lambda^{0,t}$. Therefore we can expand the conditional probability into a product of probabilities, by which we mean
\begin{equation*}
P\left(Y^X - Y^Z = k\,|\, \Lambda^{0,t}\right) = \prod_{r=1}^R P\left(\left.Y_r^X - Y_r^Z = k_r \,\right|\, \Lambda_r^{0,t}\right).
\end{equation*}
When conditioned on $\Lambda_r^{0,t}$, $Y_r^X - Y_r^Z$ is the difference of two independent Poissons with the same intensity $\Lambda_r^{0,t}$. Therefore the difference follows a Skellam distribution. To summarize,
\begin{equation*}
K_r^{0,t} \stackrel{\text{def}}{=} Y_r^X-Y_r^Z \sim \text{Skellam}(\Lambda_r^{0,t},\Lambda_r^{0,t}), \text{ when conditioned on } \Lambda^{0,t}.
\end{equation*}
Continuing from above,
\begin{equation*}
P_\nu(X_{11}(t) = X_{12}(t)) = \!\!\!\! \sum_{k\in\text{null}(S)} \!\!\!\! \mathbb{E}_{\nu,0}\left[\prod_{r=1}^R P\left(\left.K_r^{0,t} = k_r \,\right|\, \Lambda_r^{0,t}\right)\right],\\
\end{equation*}
where the expectation is taken over $\Lambda^{0,t}$.

If we are estimating a marginal distribution, then we need to modify the sum slightly. Let $S'$ be the same as $S$, except the rows corresponding to the marginalized-out variables are removed. Then replace $\text{null}(S)$ with $\text{null}(S')$.
 
\subsection{Proof of \texorpdfstring{\cref{thm:clt_integrated_mse}}{central limit theorem for the integrated mean-squared error}}
Let $\{X_i(t-h)\}_{i=1}^n$ be i.i.d.\ realizations of $X(t-h)$. Define $X_{ij}(t)$ to be the state of the CTMC conditioned on $X_{ij}(t-h) = X_i(t-h)$, where $1\le j \le m$. For simplicity, later we will denote $X_{ij}(t)$ as just $X_{ij}$.

Let $M_i \in \mathbb{Z}_{\ge 0}^{|\tilde{\mathcal{S}}|}$, where the $k$th element of $M_i$ is defined as $\sum_{j=1}^m \mathbbm{1}(X_{ij}=x_k)$. Let $\Sigma\in \mathbb{R}^{|\tilde{\mathcal{S}}| \times |\tilde{\mathcal{S}}|}$ be the covariance matrix of $M_1$. The $M_i$ are i.i.d., so if $\Sigma$ is finite, then the usual multivariate central limit theorem implies that
\begin{equation*}
\frac{1}{\sqrt{n}} \sum_{i=1}^n \left(M_i - m p_t^\nu\right) \stackrel{d}{\to} N(0,\Sigma)\text{, as } n\to\infty.
\end{equation*}
Let $M_i(x)$ denote the element if $M_i$ corresponding to $x$. Then by definition, for all $x$
\begin{equation*}
n m \hat{p}_t^\nu(x;n,m,h)= \sum_{i=1}^n M_i(x).
\end{equation*}
Therefore
\begin{equation*}
\sqrt{n} m \left(\hat{p}_t^\nu-p_t^\nu\right)\stackrel{d}{\to} N(0,\Sigma)\text{, as } n\to\infty.
\end{equation*}
The dot product is continuous, so the continuous mapping theorem implies that
\begin{equation*}
n m^2 \sum_{x\in\tilde{\mathcal{S}}}\big(\hat{p}_t^\nu(x;n,m,h) - p_t^\nu(x)\big)^2 \stackrel{d}{\to} N(0,\Sigma)^T N(0,\Sigma)\text{, as } n\to\infty.
\end{equation*}
\cite[Theorem 2.1]{box1954} implies that the right side has the same distribution as $\sum_{\ell=1}^{|\tilde{S}|}\lambda_\ell Z_\ell^2$. 
Let $\Sigma_{xx}$ be the element of $\Sigma$ on the diagonal corresponding to state $x$. Then by definition
\begin{equation*}
\begin{split}
\Sigma_{xx} &= \text{Var}\left[\sum_{j=1}^m \mathbbm{1}(X_{1j}=x) \right]\\
&= \sum_{j=1}^m \text{Var}\left[\mathbbm{1}(X_{1j}=x)\right] + 2\!\!\!\!\!\! \sum_{1\le j < k \le m} \!\!\!\!\!\! \text{Cov}\left(\mathbbm{1}(X_{1j}=x), \mathbbm{1}(X_{1k}=x) \right).
\end{split}
\end{equation*}
$\text{Var}\left[\mathbbm{1}(X_{1j}=x)\right] = p_t^\nu(x)(1-p_t^\nu(x))$, and the covariance simplifies when we rewrite it in terms of expectations. We get
\begin{equation*}
\Sigma_{xx} = m p_t^\nu(x) + m(m-1)P_\nu(X_{11}(t)=x,X_{12}(t)=x) - m^2 p_t^\nu(x)^2 < \infty.
\end{equation*}
Let $x_1$ and $x_2$ be distinct states, and let $\Sigma_{x_1,x_2}$ be the element whose row and column correspond to the states $x_1$ and $x_2$, respectively. By definition
\begin{align*}
\Sigma_{x_1,x_2} &= \text{Cov}\left[\sum_{j=1}^m \mathbbm{1}(X_{1j}=x_1), \sum_{j=1}^m \mathbbm{1}(X_{1j}=x_2)\right]\\
&= \sum_{j=1}^m\sum_{k=1}^m \text{Cov}\left[ \mathbbm{1}(X_{1j}=x_1),  \mathbbm{1}(X_{1k}=x_2)\right].
\end{align*}
Rearrange the terms in the sum to get
\begin{align*}
\sum_{j=1}^m \text{Cov}\left[ \mathbbm{1}(X_{1j}=x_1) , \mathbbm{1}(X_{1j}=x_2)\right] + \sum_{j=1}^m \sum_{\substack{k=1 \\ k\neq j}}^m  \text{Cov}\left[ \mathbbm{1}(X_{1j}=x_1),  \mathbbm{1}(X_{1k}=x_2)\right],
\end{align*}
which is equivalent to
\begin{multline*}
\sum_{j=1}^m \Big(\mathbb{E}_{\nu,0}\left[\mathbbm{1}(X_{1j}=x_1) \mathbbm{1}(X_{1j}=x_2)\right] - p(x_1)p(x_2)\Big) + \\
\sum_{j=1}^m \sum_{\substack{k=1 \\ k\neq j}}^m  \Big(\mathbb{E}_{\nu,0}\left[\mathbbm{1}(X_{1j}=x_1) \mathbbm{1}(X_{1k}=x_2)\right] - p(x_1)p(x_2)\Big).
\end{multline*}
Since $x_1\neq x_2$, $\mathbbm{1}(X_{1j}=x_1) \mathbbm{1}(X_{1j}=x_2) = 0$. Also, the second expectation can be rewritten as a probability. The above expression simplifies to
\begin{equation*}
m(m-1) P_\nu\left( X_{11}(t) = x_1, X_{12}(t) = x_2\right) -m^2 p^\nu_t(x_1)p^\nu_t(x_2) < \infty.
\end{equation*}
Equation \cref{def:sigma} simply expresses the above results with matrix-vector notation.

If we are estimating a marginal distribution, then take $\mathcal{S}$ to be the lower dimensional space corresponding to the marginal variables. Also interpret $X(t)$ as the state vector containing only the marginal variables.

\subsection{Proof of \texorpdfstring{\cref{thm:traces}}{sample variance trace theorem}}
If we write out the definition of $\hat{\Sigma}_n$ and use the fact that the trace is linear, we can see that
\begin{equation*}
\tr{\hat{\Sigma}_n} = \frac{1}{n-1} \sum_{i=1}^n \tr{\left(M_i - \bar{M}\right) \left(M_i - \bar{M}\right)^T}.
\end{equation*}
We use the cyclic property of the trace to rewrite the right side as
\begin{equation*}
\frac{1}{n-1} \sum_{i=1}^n \left(M_i - \bar{M}\right)^T \left(M_i - \bar{M}\right).
\end{equation*}
Expanding the summands leads to
\begin{equation*}
\frac{1}{n-1} \sum_{i=1}^n \left(M_i^T M_i - 2 \bar{M}^T M_i + \bar{M}^T\bar{M}\right).
\end{equation*}
From the definition of $\bar{M}$, the above expression is equal to
\begin{equation*}
- \frac{n}{n-1} \bar{M}^T\bar{M}+ \frac{1}{n-1} \sum_{i=1}^n M_i^T M_i.
\end{equation*}
By definition, $m \hat{p}_t = \bar{M}$, therefore
\begin{equation*}
\tr{\hat{\Sigma}_n} =- \frac{nm^2}{n-1} (\hat{p}_t^\nu)^T \hat{p}_t^\nu + \frac{1}{n-1} \sum_{i=1}^n M_i^T M_i.
\end{equation*}
Next consider $\tr{\hat{\Sigma}_n^2}$. We will proceed in a similar way. By definition
\begin{align*}
\hat{\Sigma}_n^2 &= \frac{1}{(n-1)^2}\left[\sum_{i=1}^n (M_i-\bar{M})(M_i-\bar{M})^T \right]^2\\
&= \frac{1}{(n-1)^2} \sum_{i=1}^n \sum_{j=1}^n (M_i-\bar{M})(M_i-\bar{M})^T (M_j-\bar{M})(M_j-\bar{M})^T.
\end{align*}
The trace is linear, so
\begin{align*}
\tr{\hat{\Sigma}_n^2} &= \frac{1}{(n-1)^2} \sum_{i=1}^n \sum_{j=1}^n \tr{(M_i-\bar{M})(M_i-\bar{M})^T (M_j-\bar{M})(M_j-\bar{M})^T}\\
&= \frac{1}{(n-1)^2} \sum_{i=1}^n \sum_{j=1}^n \left[(M_i-\bar{M})^T(M_j-\bar{M})\right]^2.
\end{align*}
The last line follows from the cyclic property of the trace. When we expand the summands, the right side becomes
\begin{equation*}
\frac{1}{(n-1)^2} \sum_{i=1}^n \sum_{j=1}^n \left[ M_i^T M_j - \bar{M}^T M_i - \bar{M}^T M_j  + m^2 (\hat{p}_t^\nu)^T \hat{p}_t^\nu \right]^2.
\end{equation*}
As for the claim about almost sure convergence of the traces, note that $\hat{\Sigma}_n\stackrel{\text{a.s.}}{\to} \Sigma$. Since matrix multiplication and the trace are continuous, the continuous mapping theorem implies the result.

\subsection{Proof of \texorpdfstring{\cref{thm:upper_confidence_bound}}{confidence interval theorem}}
Define
\begin{equation*}
U = \frac{\tr{\Sigma^2}}{\tr{\Sigma}}\chi_\alpha^2\left(\frac{\tr{\Sigma}^2}{\tr{\Sigma^2}} \right).
\end{equation*}
Since $\hat{\Sigma}_n \stackrel{\text{a.s.}}{\to} \Sigma \text{ as } n\to\infty$, the continuous mapping theorem and \cref{thm:un_to_u} taken together imply that $U_n\to U$ almost surely as $n\to \infty$. Also \cref{thm:clt_integrated_mse} says that
\begin{equation*}
nm^2 \sum_{x\in\tilde{\mathcal{S}}} \big(\hat{p}_t^\nu(x;n,m,h)-p_t^\nu(x)\big)^2 \stackrel{d}{\to} \sum_{\ell=1}^{|\tilde{S}|} \lambda_\ell Z_\ell^2, \text{ as } n\to\infty.
\end{equation*}
Therefore by Slutsky's theorem
\begin{equation*}
\frac{nm^2 \sum_{x\in\tilde{\mathcal{S}}} \big(\hat{p}_t^\nu(x;n,m,h)-p_t^\nu(x)\big)^2}{U_n} \stackrel{d}{\to} \frac{\sum_{\ell=1}^{|\tilde{S}|} \lambda_\ell Z_\ell^2}{U}, \text{ as } n\to\infty,
\end{equation*}
which we can rewrite as
\begin{equation*}
\lim_{n\to\infty} P_\nu\left(nm^2 \sum_{x\in\tilde{\mathcal{S}}} \big(\hat{p}_t^\nu(x;n,m,h)-p_t^\nu(x)\big)^2 \le U_n \right) = P\left(\sum_{\ell=1}^{|\tilde{S}|} \lambda_\ell Z_\ell^2 \le U\right).
\end{equation*}
Applying the Satterthwaite approximation \cite{satterthwaite1941} to the right-hand side gives
\begin{align*}
&\lim_{n\to\infty} P_\nu\left(nm^2\sum_{x\in\tilde{\mathcal{S}}} \big(\hat{p}_t^\nu(x;n,m,h)-p_t^\nu(x)\big)^2 \le U_n \right)\\
&\quad\quad\approx P\left(\frac{\tr{\Sigma^2}}{\tr{\Sigma}} \chi^2 \left(\frac{\tr{\Sigma}^2}{\tr{\Sigma^2}} \right) \le U \right)\\
&\quad\quad= 1-\alpha.
\end{align*}
The result still holds for marginal distributions. We just need to remove the coordinates of $\tilde{\mathcal{S}}$ corresponding to the variables that are marginalized out.

\begin{lemma}\label{thm:un_to_u}
Let $X_\theta$ be a family of random variables parameterized by $\theta \in \mathbb{R}$ with strictly increasing cumulative distribution functions $F_\theta$.  Suppose that for each $\theta$, the function $F_\theta$ is continuous. Assume also that  $F_\theta(x)$ is continuous in $\theta$ for each $x \in \mathbb{R}$.  Then the $1-\alpha$ quantiles of $F_\theta$ are also continuous in $\theta$ for all $\alpha\in(0,1)$.
\end{lemma}
\begin{proof}
Let $\alpha\in(0,1)$, and let $\{\theta_n\}_{n=1}^\infty$ be a sequence that converges to $\theta$. Define $q_n$ and $q$ to be the $1-\alpha$ quantiles corresponding the $\theta_n$ and $\theta$, respectively. We want to show that $q_n$ converges to $q$.

 Let $\varepsilon > 0$. Since $\alpha \in (0,1)$, we know that $q$ is finite.  Therefore, we can choose $\underline{q}$ and $\overline{q}$ such that
\begin{equation*}
\underline{q} < q < \overline{q} \quad \text{ and } \quad \overline{q} - \underline{q} < \varepsilon.
\end{equation*}
We want to show that $|q_n - q| < \varepsilon$ for all sufficiently large $n$, so it will suffice to prove that $\underline{q} < q_n < \overline{q}$ for all $n$ large enough.

By assumption, $F_\theta(\underline q)$ is continuous in $\theta$, so
\begin{equation*}
\lim_{n\to\infty} F_{\theta_n}(\underline{q}) = F_\theta(\underline{q}) < F_\theta(q) = 1-\alpha = F_{\theta_n}(q_n).
\end{equation*}
The inequality is strict, because $q$ is a quantile and $F_\theta$ is strictly increasing and $\underline{q} < q$.
Since $F_{\theta_n}$ is non--decreasing, $q_n > \underline{q}$ for all sufficiently large $n$. We can use essentially the same argument to conclude that $q_n < \overline{q}$ for all $n$ large enough.
\end{proof}

\section{Expectations}\label{sec:expectations}

The specific conditional Monte Carlo method introduced in this paper has been developed to estimate the entire distribution in a manner that is more efficient than regular Monte Carlo, as quantified by the mean  integrated squared error \cref{eq:tominimize} for a fixed computational budget.  This does not imply that it will be more efficient in the computation of any specific expectation.  In fact, in this Appendix we prove that it is necessarily less efficient in computing the first moment of a linear birth model.  Specifically, we prove that for a fixed computational budget the variance of the estimator generated via the conditional Monte Carlo method is greater than or equal to the variance of the standard Monte Carlo estimator.  This demonstrates that caution is required when  implementing a method in a context it was not   intended for.

Recall the Birth Model, which consists of the single reaction $X\xrightarrow{1} 2X,$
where we have chosen a rate parameter of 1.  Assuming a fixed initial condition of $X_0\in \Z_{\ge0}$, it is straightforward to show that
\[
	\mathbb{E}[X(t)] = X_0 e^t \quad \text{and}\quad \text{Var}[X(t)] = X_0 e^t(e^t - 1).
\]


For a fixed number of paths $n_1$, and a point mass $X_0$, the standard Monte Carlo estimator has an expected cost--quantified by the number of random variables utilized--of
\[
	\text{Cost}_\text{MC}(n_1) = \mathbb{E}[ n_1 ( X(t) - X_0)] = n_1 X_0 ( e^t - 1),
\]
and a variance of $\text{Var}\left[n_1^{-1} \sum_{i=1}^{n_1} X_i(t)\right] = n_1^{-1} \text{Var}(X_1(t))= n_1^{-1} X_0 e^t(e^t - 1)$.

For a fixed number of paths $n$ and $m$, and a fixed parameter $h \in [0,t]$, the expected cost of the conditional Monte Carlo estimator is
\begin{align*}
	\text{Cost}_\text{CMC} (n,m,h) &= n\,\mathbb{E}[X_{1i}(t-h) - X_0] + n\cdot m\,\mathbb{E}[X_{1i}(t) - X_{1i}(t-h)]\\
	&= nX_0 (e^{t-h} - 1) + n\cdot m X_0 (e^{t}- e^{t-h}).
\end{align*}
The variance of the conditional Monte Carlo estimator is
\begin{align}\label{eq:8970978}
	\text{Var}\left[ \frac1n \frac1m \sum_{i=1}^n \sum_{j = 1}^m X_{ij}(t) \right] &= \frac1{n\cdot m^2} \text{Var}\left[ \sum_{j = 1}^m X_{ij}(t) \right].
\end{align}
Using the generic result that for random variables $X$ and $Y$ on the same probability space $\text{Var}(X) = \mathbb{E}[\text{Var}(X|Y)] + \text{Var}(\mathbb{E}[X|Y])$, we have
\begin{align*}
\text{Var}\left( \sum_{j = 1}^m X_{ij}(t) \right) &= \mathbb{E}\left[\text{Var}\left(\sum_{j = 1}^m X_{1j}(t) \bigg| X_{11}(t-h)\right)\right] \\
&\hspace{.2in} + \text{Var}\left( \mathbb{E}\left[ \sum_{j=1}^m X_{1j}(t) \bigg| X_{11} (t-h) \right]\right)\\
&= m\,\mathbb{E}[ \text{Var}(X_{1j}(t) | X_{11}(t-h)) ]+ \text{Var}(m\,\mathbb{E}[X_{1j}(t)|X_{11}(t-h)])\\
&= m\,\mathbb{E} \left[ X_{11}(t-h) e^h(e^h-1)  \right] + m^2 \text{Var}\left( X_{11}(t-h) e^h\right)\\
&= m X_0e^{t-h} e^h(e^h-1) + m^2 e^{2h} X_0e^{t-h}(e^{t-h}-1)\\
&= m X_0 e^t(e^h-1) + m^2 X_0 e^t(e^t - e^h)
\end{align*}
Thus, dividing by $n\cdot m^2$ as in \cref{eq:8970978}, the variance of the conditional Monte Carlo estimator is
\[
	\text{Var}\left[ \frac1n \frac1m \sum_{i=1}^n \sum_{j = 1}^m X_{ij}(t) \right] =\frac1{n\cdot m}\left[ X_0 e^t(e^h-1) + m X_0e^t(e^t-e^h)\right].
\]
For a fixed $n_1$, setting $\text{Cost}_\text{CMC}(n,m,h) = \text{Cost}_\text{MC}(n_1)$ yields
\begin{align*}
nX_0 (e^{t-h} - 1) + n\cdot m X_0 (e^{t}- e^{t-h}) = n_1 X_0 ( e^t - 1),
\end{align*}
or
\[
	n = \frac{n_1  ( e^t - 1)}{ (e^{t-h} - 1) + m  (e^{t}- e^{t-h}) }.
\]
Thus, for a fixed $n_1$ and $n$ chosen above the variance of the conditional Monte Carlo estimator is
\begin{align*}
	\frac1{n_1} \cdot \frac{(e^{t-h} - 1) + m  (e^{t}- e^{t-h})}{m (e^t-1)}\left[ X_0 e^t(e^h-1) + m X_0e^t(e^t-e^h)\right].
\end{align*}
This is minimized at the boundary with  $m = 1$, giving exactly the same variance as the regular Monte Carlo estimator.  Thus, to summarize, for a given fixed computational cost the variance of the conditional Monte Carlo estimator must be larger than the variance of the standard Monte Carlo estimator.

\section*{Acknowledgments}
We are grateful for financial support from the Army Research Office through grant W911NF-18-1-0324 and the National Science Foundation through grant DMS-2051498.

\bibliographystyle{siamplain}
\bibliography{conditional}
\end{document}


\maketitle

\section{Objective function figures}
This section contains figures related to 
the objective function of the optimization problem (3.10).


\begin{figure}[ht]
	\centering
	\includegraphics[width=\textwidth]{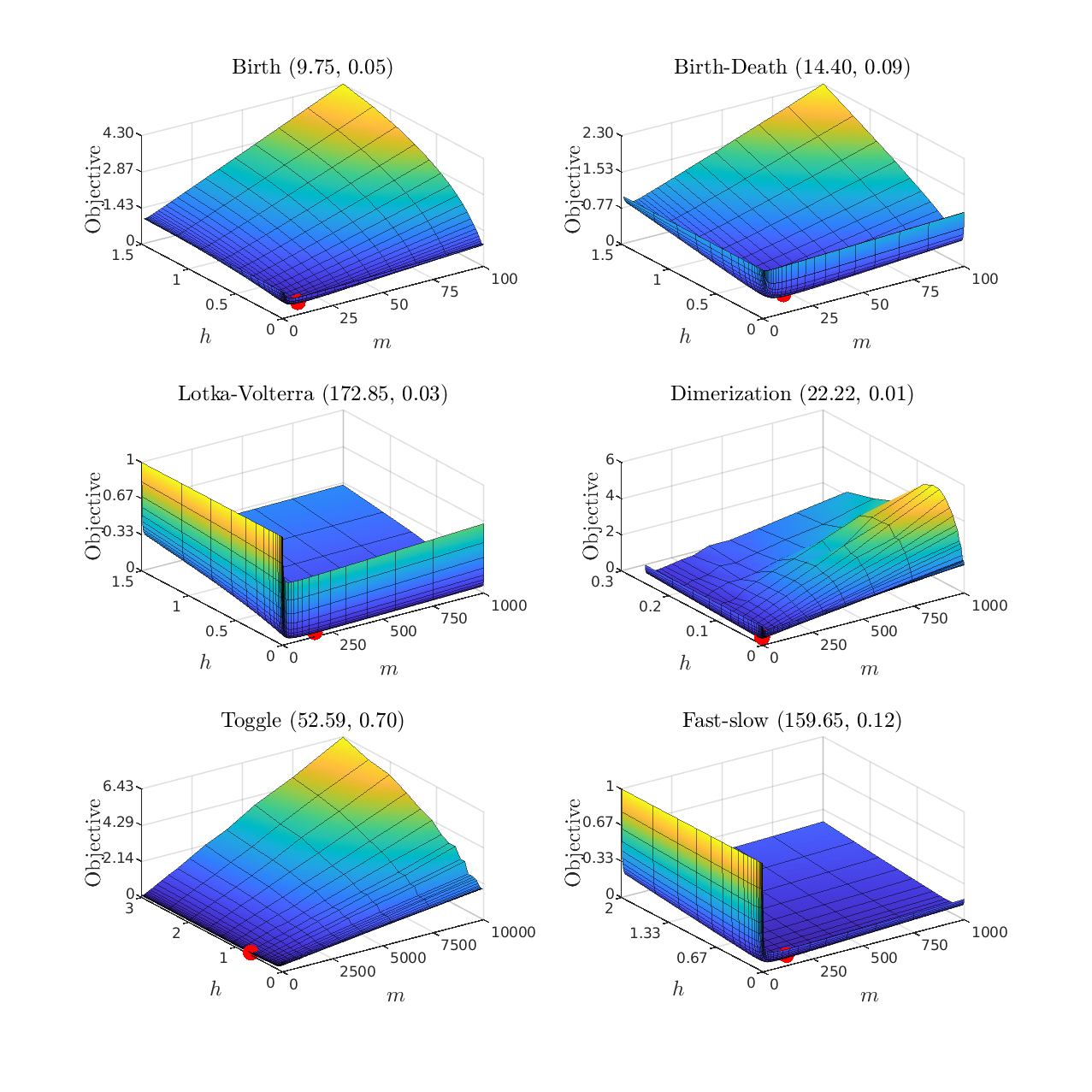}
	\caption{The surfaces are the approximate objective functions $\hat{f}(m,h)$ for the different example models. The red dots show where the minimums are achieved, and the pair of numbers in each title are the coordinates of the minimum in $m$ and $h$ space. The approximate objective function utilized the approximation  $\hat{p}^2=0$. For $\tilde{N}$, we took all possible linear combinations of the nullspace basis vectors with coefficient over $\{-4, \ldots, 4\}$. We used $10^3$ samples to estimate the parameters in the objective function. We used MATLAB's \textit{fminsearch} function (a derivative-free optimizer) to find a minimizer of $\hat{f}(m,h)$.}
	\label{fig:objective_functions}
\end{figure}

\begin{figure}[ht]
\centering
\includegraphics[width=\textwidth]{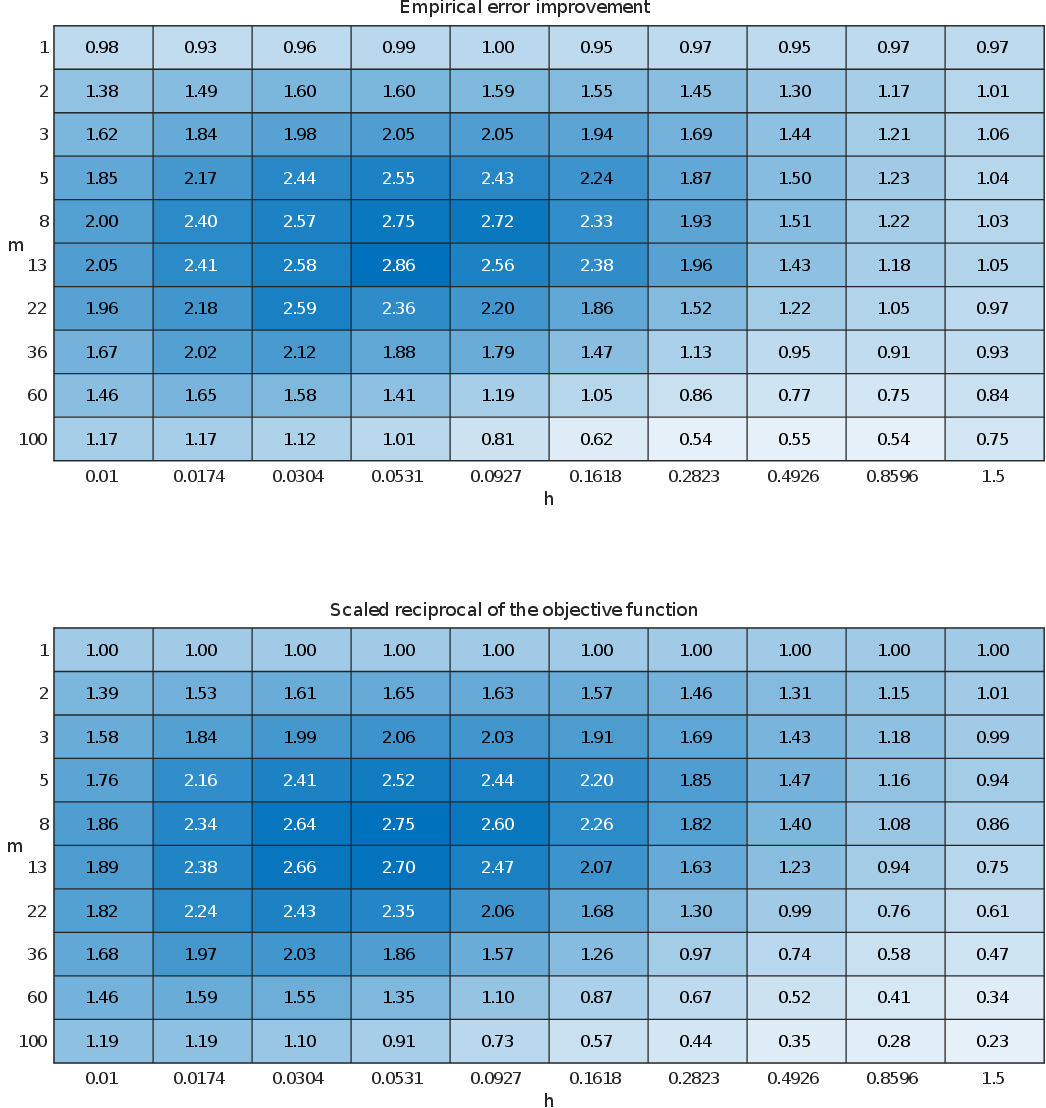}
\caption{Birth model. The first heatmap shows $\text{MISE}_\text{MC} / \text{MISE}_\text{CMC}(m,h)$ for different values of $m$ and $h$. The method we used to obtain the ratio is described in section 4. The second heatmap shows that value of $\hat{f}(1,0)/\hat{f}(m,h)$. The definition of $\hat{f}$ is in (3.9)}
\label{fig:heatmap_birth}
\end{figure}

\begin{figure}[ht]
\centering
\includegraphics[width=\textwidth]{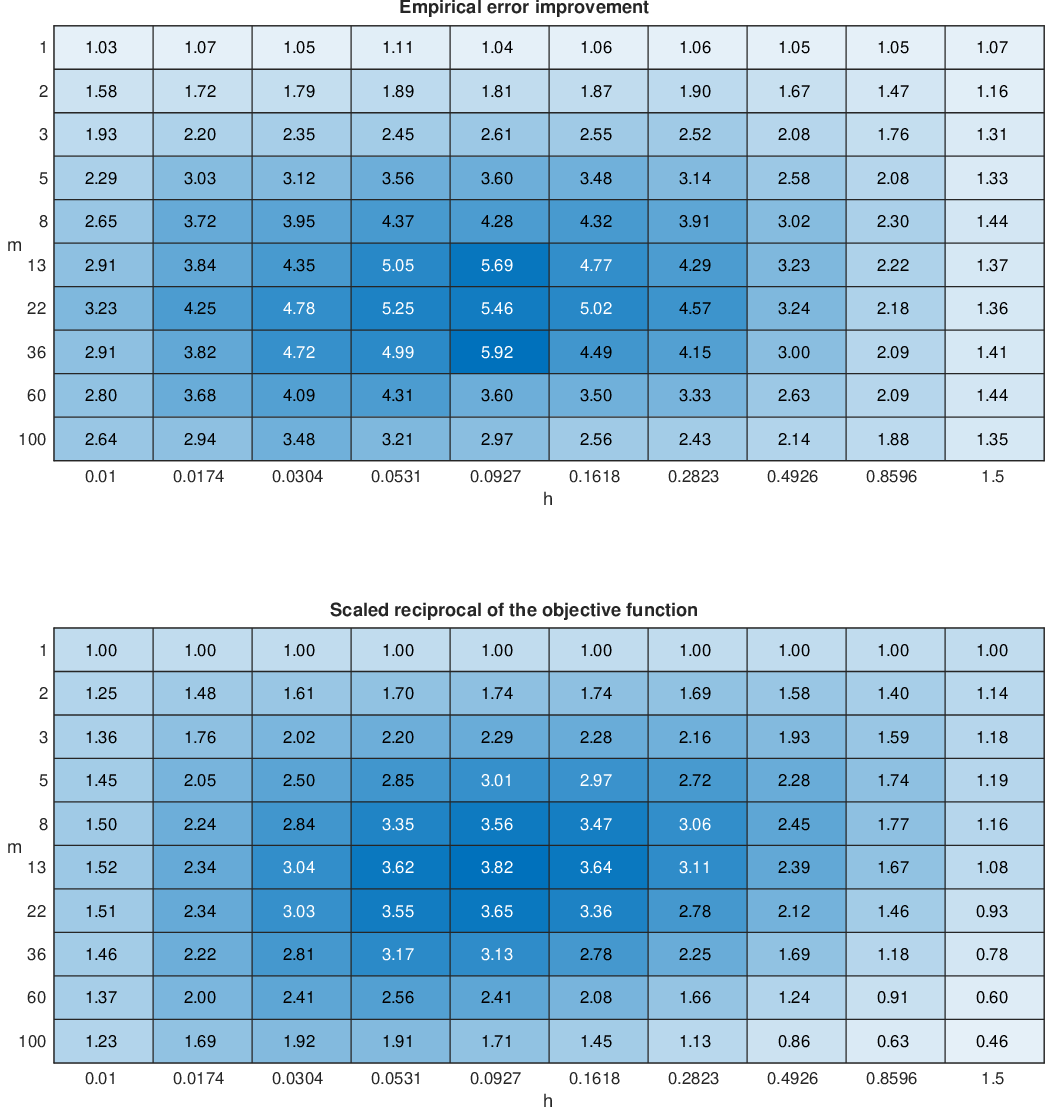}
\caption{Birth-death model. The first heatmap shows $\text{MISE}_\text{MC} / \text{MISE}_\text{CMC}(m,h)$ for different values of $m$ and $h$. The method we used to obtain the ratio is described in section 4. The second heatmap shows that value of $\hat{f}(1,0)/\hat{f}(m,h)$. The definition of $\hat{f}$ is in (3.9).}
\label{fig:heatmap_birth_death}
\end{figure}

\begin{figure}[ht]
\centering
\includegraphics[width=\textwidth]{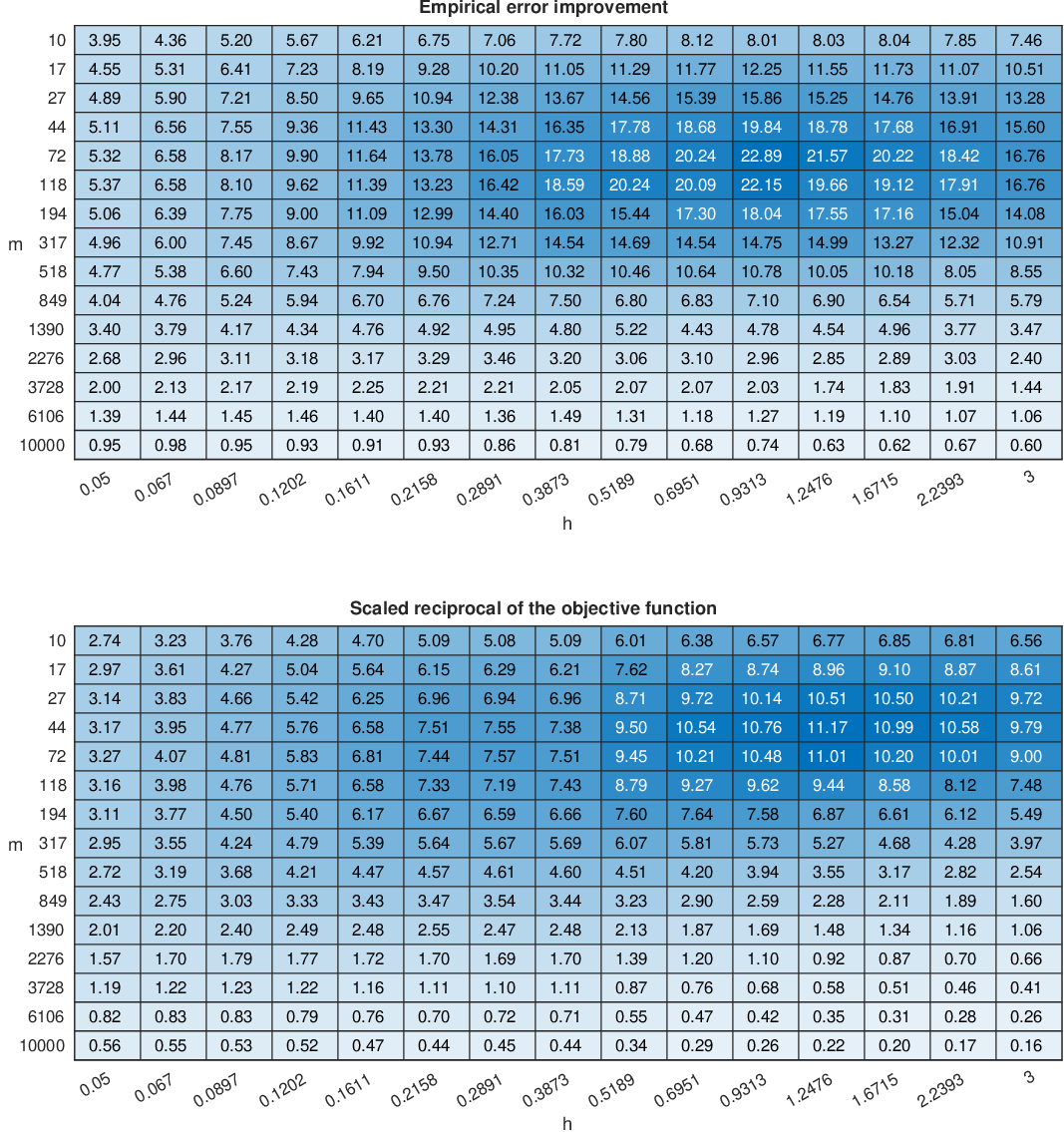}
\caption{Gene toggle model. The first heatmap shows $\text{MISE}_\text{MC} / \text{MISE}_\text{CMC}(m,h)$ for different values of $m$ and $h$. The method we used to obtain the ratio is described in section 4. The second heatmap shows that value of $\hat{f}(1,0)/\hat{f}(m,h)$. The definition of $\hat{f}$ is in (3.9).}
\label{fig:heatmap_toggle}
\end{figure}

\begin{figure}[ht]
\centering
\includegraphics[width=\textwidth]{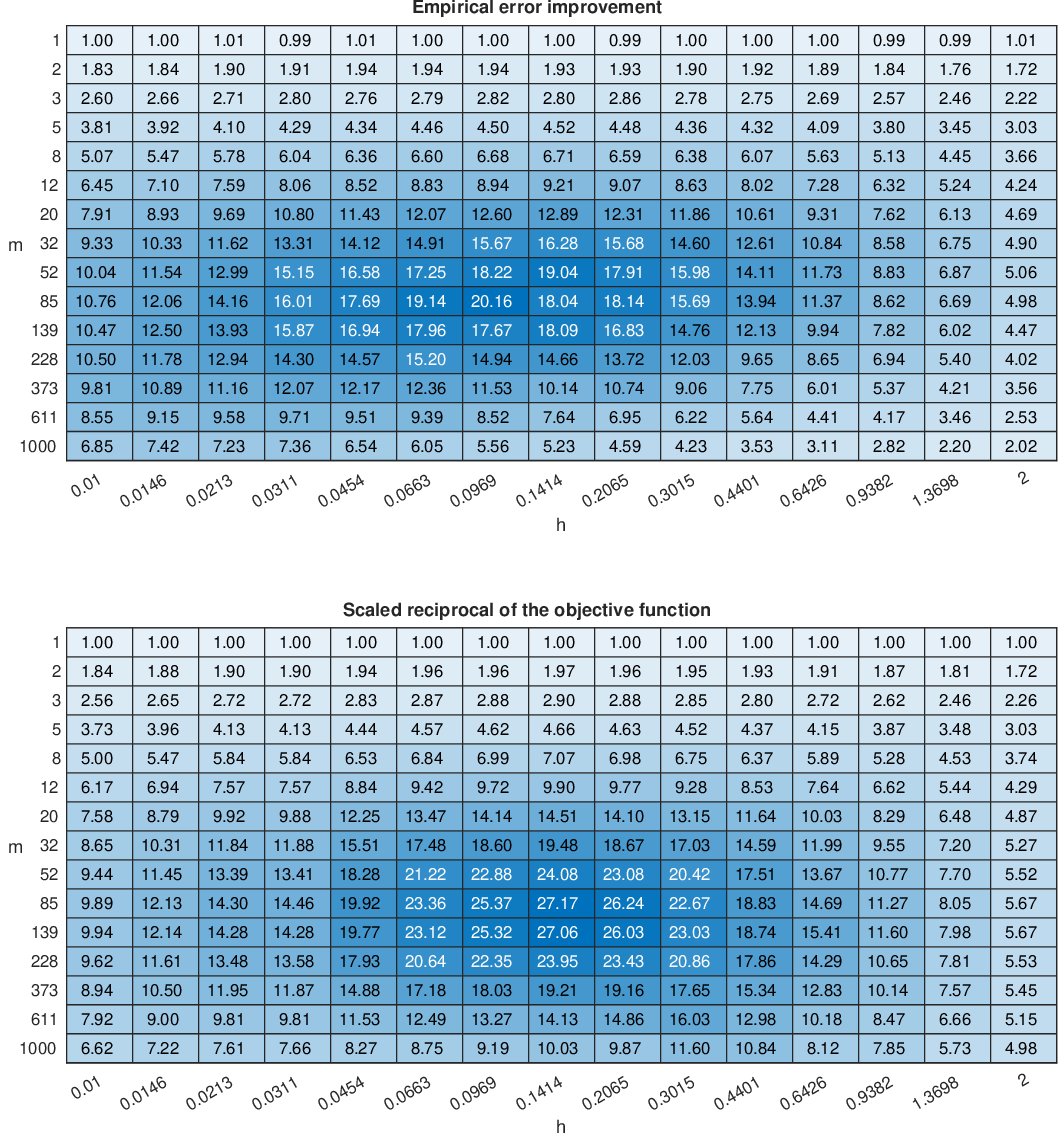}
\caption{Fast/slow model. The first heatmap shows $\text{MISE}_\text{MC} / \text{MISE}_\text{CMC}(m,h)$ for different values of $m$ and $h$. The method we used to obtain the ratio is described in section 4. The second heatmap shows that value of $\hat{f}(1,0)/\hat{f}(m,h)$. The definition of $\hat{f}$ is in (3.9).}
\label{fig:heatmap_fast_slow}
\end{figure}

\clearpage

\section{Central limit theorem figures}
This section contains figures related to the central limit theorem (Theorem 5.1). The dashed blue density is the empirical density of the integrated squared error (5.1), and the solid red density is the Satterwaithe approximation to the asymptotic density (5.4). On the x-axis there is a blue cross and red circle. The blue cross is the 95\% sample quantile, and the red circle is the 95\% quantile of the red density. To generate the blue curve, we sampled $10^4$ values of $nm^2\sum_{x\in\tilde{\mathcal{S}}} \big(\hat{p}_t(x;n,m,h)-p_t(x)\big)^2$ for different values of $n$, and then we gave those values to MATLAB's \textit{ksdensity} function. The traces of $\Sigma$ and $\Sigma^2$ were based an independent set of $10^5$ simulations. Since the densities and quantiles agree, the red density and its corresponding 95\% quantile are a good approximation. Therefore we can use them to provide confidence intervals for the integrated squared error.

\begin{figure}[htbp]
    \centering
    \includegraphics[width=\textwidth]{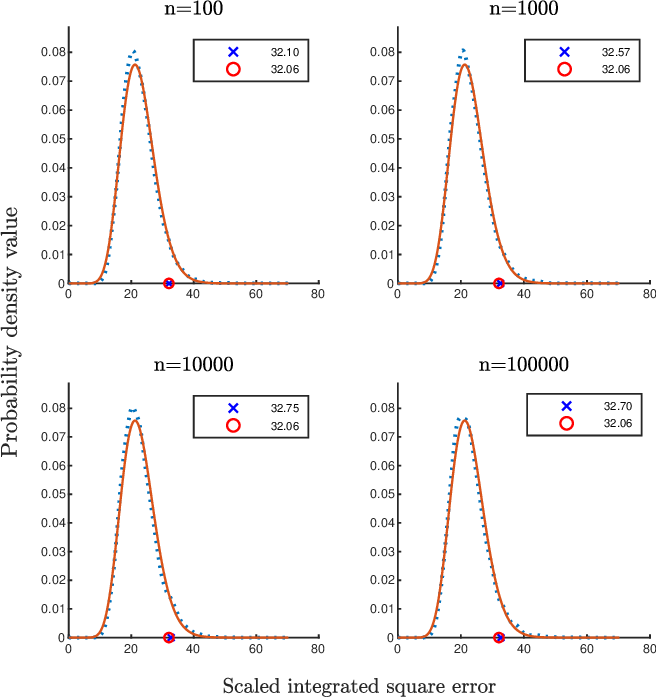}
    \caption{Birth model.}
    \label{fig:clt_birth}
\end{figure}

\begin{figure}[htbp]
    \centering
    \includegraphics[width=\textwidth]{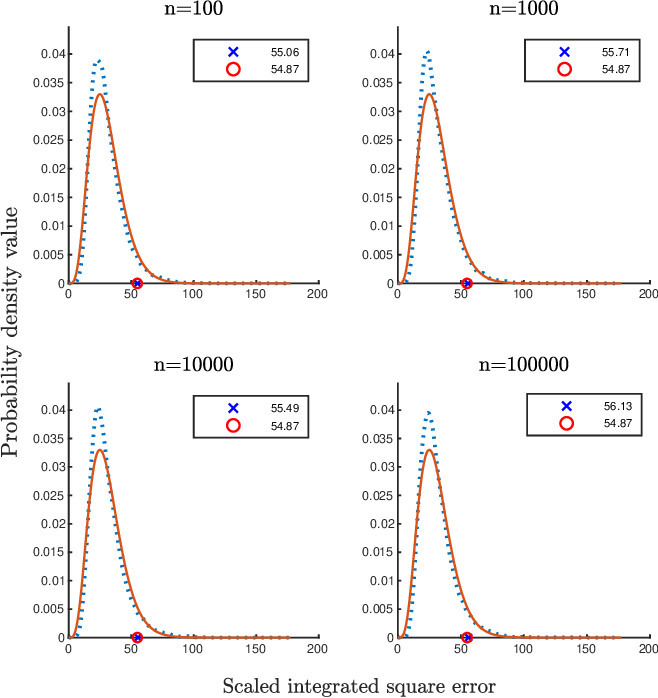}
    \caption{Birth-death model.}
    \label{fig:clt_birth_death}
\end{figure}

\begin{figure}[htbp]
    \centering
    \includegraphics[width=\textwidth]{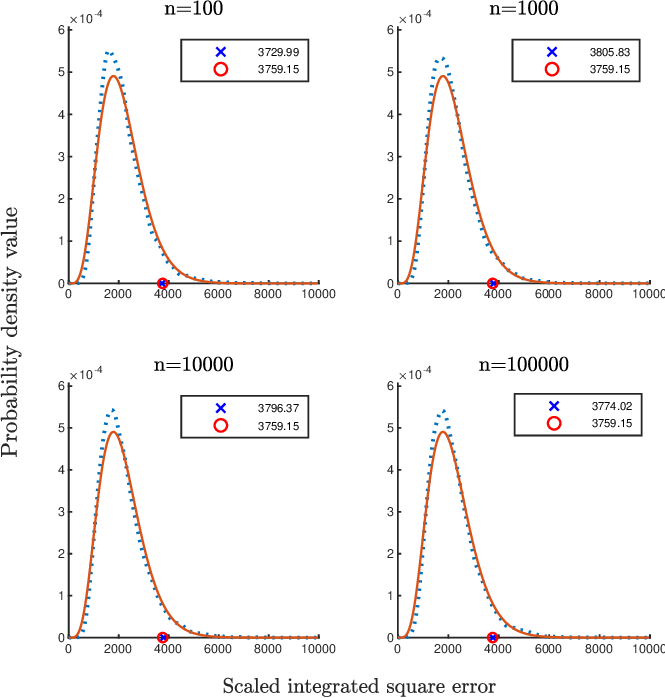}
    \caption{Toggle model.}
    \label{fig:clt_toggle}
\end{figure}

\begin{figure}[htbp]
    \centering
    \includegraphics[width=\textwidth]{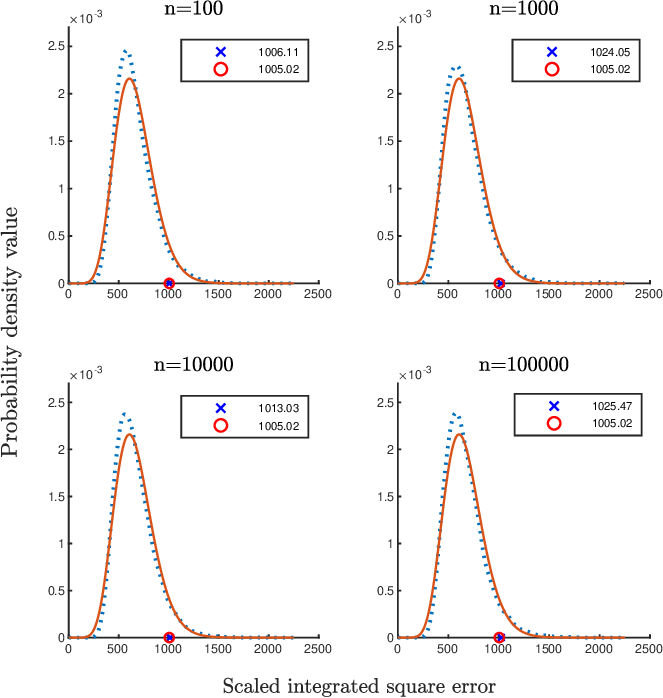}
    \caption{Fast/slow model.}
    \label{fig:clt_fast_slow}
\end{figure}